\documentclass{article}

\usepackage{arxiv}

\usepackage[utf8]{inputenc} 
\usepackage[T1]{fontenc}    
\usepackage{hyperref}       
\usepackage{url}            
\usepackage{booktabs}       
\usepackage{amsfonts}       
\usepackage{nicefrac}       
\usepackage{microtype}      
\usepackage{lipsum}		
\usepackage{graphicx}
\usepackage{tikz}
\usepackage{caption}

\usepackage{doi}

\usepackage{amsmath}
\usepackage{amssymb}
\usepackage{amsthm} 
\usepackage{gensymb}
\usepackage{mathtools}
\usepackage{cancel} 
\usepackage{nicefrac}
\usepackage{upgreek} 
\usepackage{enumitem} 
\usepackage{bookmark}
\usepackage{listings} 
\usepackage{epstopdf} 
\usepackage{mathrsfs} 
\usepackage{mathdots} 
\usepackage{longtable} 
\usepackage[toc,page]{appendix}
\usepackage{color}
\usepackage{multibib}

\usepackage{algorithm}
\usepackage[noend]{algpseudocode} 

\usepackage{subcaption} 
\usepackage{mwe}

\usepackage{booktabs}

\title{Ordinal Optimisation and the Offline Multiple Noisy Secretary Problem}

\author{ Robert Chin\thanks{Department of Electrical and Electronic Engineering, The University of Melbourne, Australia \& School of Computer Science, University of Birmingham, United Kingdom. Email: \texttt{r.chin4@student.unimelb.edu.au}}
	\And
	Jonathan E. Rowe\thanks{School of Computer Science, University of Birmingham, United Kingdom \& The Alan Turing Institute, United Kingdom}
	\And
	Iman Shames\thanks{College of Engineering \& Computer Science, Australian National University, Australia} \\
	\And
	Chris Manzie\thanks{Department of Electrical and Electronic Engineering, The University of Melbourne, Australia} \\
	\And
	Dragan Ne\v{s}i\'{c}\thanks{Department of Electrical and Electronic Engineering, The University of Melbourne, Australia}
}

\date{}

\renewcommand{\headeright}{}
\renewcommand{\undertitle}{}

\numberwithin{equation}{section} 

\makeatletter
\newcommand{\pushright}[1]{\ifmeasuring@#1\else\omit\hfill$\displaystyle#1$\fi\ignorespaces}
\newcommand{\pushleft}[1]{\ifmeasuring@#1\else\omit$\displaystyle#1$\hfill\fi\ignorespaces}
\makeatother

\newtheorem{problem}{Problem}[section]
\newtheorem{definition}{Definition}[section]
\newtheorem{proposition}{Proposition}[section]

\newtheorem{theorem}{Theorem}[section]
\newtheorem{lemma}{Lemma}[section]

\begin{document}
\maketitle

\begin{abstract}
We study the success probability for a variant of the secretary problem, with noisy observations and multiple offline selection. Our formulation emulates, and is motivated by, problems involving noisy selection arising in the disciplines of stochastic simulation and simulation-based optimisation. In addition, we employ the philosophy of ordinal optimisation - involving an ordinal selection rule, and a percentile notion of goal softening for the success probability. As a result, it is shown that the success probability only depends on the underlying copula of the problem. Other general properties for the success probability are also presented. Specialising to the case of Gaussian copulas, we also derive an analytic lower bound for the success probability, which may then be inverted to find sufficiently large sample sizes that guarantee a high success probability arbitrarily close to one.
\end{abstract}

\keywords{Ordinal optimization \and Secretary problem \and Copulas}

\section{Introduction}

Numerous problems studied in operations research are themed around the objective of `picking the best' from a finite number of alternatives, under uncertainty. The branch of \textit{ranking and selection} problems \cite{Dudewicz1980} typically considers finding the population with the best mean from a collected sample. Online (i.e. sequential) variations of ranking and selection have also appeared, e.g. \cite{Peng2018}, which are also closely related to the \textit{pure exploration} (best-arm identification) setting in multi-armed bandits \cite{Bubeck2009}. Common motivators for this line of work include stochastic search/simulation \cite{Fu2014} and the design of clinical trials \cite{Shen2020}. \\

On the other hand, the \textit{secretary problem} (also known as the "best choice" problem) is an optimal stopping problem that is commonly framed in the context of selecting the best candidate for a job, when the quality of interviewed candidates are observed sequentially in random order \cite{Bruss1984}. The well-known asymptotically optimal solution for maximising the probability of best selection features the reciprocal of Euler's number (approximately 0.37) as a stopping rule \cite[Equation (2a-10)]{Gilbert1966}. Variations of the secretary problem are abundant in the literature, for instance when multiple candidates may be selected \cite{Preater1994}, or when observations of candidate quality are corrupted by noise \cite{Krieger2012}. \\

Ordinal optimisation (OO) is another paradigm, originally introduced in \cite{Ho1992}, that is proposed for softening difficult problems in stochastic search and optimisation, motivated by the simulation-based design of discrete-event dynamical systems. It is offered as a complementary approach to conventional optimisation techniques whenever there is `little hope' of finding the global optimum solution, and its operation rests on two underlying principles: 1) by selecting the subset according to order, the selection is more `robust' to noise; and 2) by \textit{goal softening} (i.e. increasing the degree of sub-optimality), chances of success can be improved. \\

In the offline (i.e. non-sequential) formulation of OO \cite[Chapter II]{Ho2007}, the technique boils down to "randomly sample many candidate solutions, simulate their performances, and pick the best observed". This is also known as the \textit{horse-race} selection rule, which is formally proven to be optimal under the setup considered in \cite[Theorem 3.1]{Yang2002}. The outcome provided by OO is a high probability guarantee that one or more out of a selected subset of candidate solutions is an acceptable sub-optimal solution. It is interesting to note that the same philosphy of "sample many candidate solutions and pick the best observed" also independently manifests itself in a corner of the probabilistically robust control literature \cite{Vidyasagar2001}, justified by rigorous sample complexity analysis. A research area with roots from OO is the optimal computing budget allocation (OCBA) framework, which addresses the problem of efficient sequential allocation of simulation resources \cite{Chen2000}. This work on OCBA can also be regarded as an early precursor to results in online ranking and selection, and pure exploration bandits \cite[\S 33.5]{Lattimore2020}. \\

In this paper, we investigate a secretary-like problem that is motivated by simulation-based optimisation. Specifically, we consider contexts where noisy offline simulations are relatively cheap/plentiful compared to `true' evaluations of performance, which are relatively expensive/limited. This may be represented, for example, by the learning of a robotic controller that is primarily outsourced to a computer simulation prior to physical deployment, under the \textit{sim-to-real} methodology \cite{Peng2018a}. To this end, we introduce the framework of an offline noisy secretary problem, and regard the probability of success as the probability in which \textbf{at least one selected candidate performs acceptably well}. In addition, we employ in our formulation the two cornerstones championed by OO: an ordinal horse-race selection rule with fixed selection size for multiple secretaries, and goal softening as another degree of freedom to control the probability of success. \\

We study the properties of the success probability in general settings. In particular, our formulation leads to a success probability which only depends on the \textit{bivariate copula} of the underlying distributions. Note that while copula modelling has been previously applied in stochastic simulation \cite{Biller2009}, it was used for a different purpose (namely, to model the time-series of simulation inputs). By keeping to the motivational context that the allowable number of offline evaluations is plentiful, we also address the \textit{inversion} problem, that is to determine a sufficiently large sample size of candidates, so that for fixed selection size, \textbf{a prescribed high probability of success is guaranteed}. \\

The paper is structured as follows. In Section \ref{sec:prelim}, we setup the technical premise of the paper, and define the success probability for the offline multiple noisy secretary problem. The inversion problem is also to be stated formally. In Section \ref{sec:OO_properties}, we provide an explicit integral expression for the success probability that only depends on the underlying copula, and use this to illustrate various properties of the success probability for general copula models. In Section \ref{sec:gaussian_copula_ordinal_optimisation}, the success probability is specialised to the case of Gaussian copula models, in which we provide an analytic lower bound for the success probability. This analytic lower bound is then used to address the inversion problem, where we demonstrate finding sufficiently large sample sizes to guarantee a high probability of success. Concluding remarks are given and future work is proposed in Section \ref{sec:conclusion}.

\section{Preliminaries}
\label{sec:prelim}

\subsection{Notation}
\label{sec:notation}
The set $\mathbb{R}$ denotes the real numbers, and $\mathbb{N}$ denotes the set of natural numbers. The logarithm $\log\left(\cdot\right)$ is taken to mean the natural logarithm, while $\cot\left(\cdot\right)$ and $\sinh\left(\cdot\right)$ are the cotangent and hyperbolic sine functions respecrively. The floor and ceiling operators are given by $\left\lfloor\cdot\right\rfloor$ and $\left\lceil\cdot\right\rceil$ respectively. Whenever the symbols $\leq$ and $<$ are used between vectors, they refer to element-wise inequalities. The standard Gaussian probability density function (PDF), cumulative distribution function (CDF), inverse CDF (i.e. quantile function) and $Q$-function (complementary CDF) are canonically represented using $\phi\left(\cdot\right)$, $\Phi\left(\cdot\right)$, $\Phi^{-1}\left(\cdot\right)$, $Q\left(\cdot\right)$ respectively. We write $\mathbf{X} \sim \mathcal{N}\left(\boldsymbol{\mu}, C\right)$ to denote that $\mathbf{X}$ is Gaussian distributed with mean $\boldsymbol{\mu}$ and covariance $C$. The probability of an event is measured by $\operatorname{Pr}\left(\cdot\right)$ with respect to a probability space that is clear from context. The abbreviation i.i.d. stands for mutually independent and identically distributed. The symbol $\mathbb{E}$ denotes mathematical expectation. We use $\mathbb{I}_{\left\{\cdot\right\}}$ to denote an indicator variable. The notation $\left[X \middle| Y = y\right]$ is understood to mean a random variable that is equal in law to the conditional distribution of $X$ given $Y = y$. Following the notation of \cite{Shaked2007}, stochastic dominance is denoted by $\underset{\mathrm{st}}{\preceq}$ and the symbol $\underset{\mathrm{st}}{=}$ between random elements denotes equality in law. The $k$\textsuperscript{th} order statistic of a i.i.d. sample of size $n$ from parent distribution $Z$ will be denoted by $Z_{k:n}$. 

\subsection{Setup}

Let the pair $\left(Z, X\right)$ have a continuous distribution with joint CDF $F_{Z,X}\left(z, x\right)$ and respective marginal CDFs $F_{Z}\left(z\right)$, $F_{X}\left(x\right)$. Consider $n$ i.i.d. copies $\left(Z_{i}, X_{i}\right)$ drawn from the distribution of $\left(Z, X\right)$. The variables $Z_{1}, \dots, Z_{n}$ (i.e. noisily observed candidate qualities) are ordered from best to worst, denoted by $Z_{1:n} \leq \dots \leq Z_{n:n}$. The best $m$ (with $m \leq n$) are selected, given by $Z_{1:n} , \dots, Z_{m:n}$, with respectively attached $X$-values (i.e. actual candidate qualities) denoted as $X_{\left\langle 1\right\rangle}, \dots, X_{\left\langle m\right\rangle}$. More explicitly, the pairs $\left(Z_{1:n}, X_{\left\langle 1\right\rangle}\right), \dots, \left(Z_{m:n}, X_{\left\langle m\right\rangle}\right)$ have been chosen. For this offline multiple noisy secretary problem, we define the corresponding success probability as follows.

\begin{definition}[Success probability]
The success probability $p_{\mathrm{success}}$ is the probability that \textbf{at least one} of $X_{\left\langle 1\right\rangle}, \dots, X_{\left\langle m\right\rangle}$ is in the best $100\alpha$ percentile of the population of $X$, and denoted
\begin{align}
p_{\mathrm{success}}^{F_{Z,X}}\left(n, m, \alpha\right) &:= \operatorname{Pr}\left(\bigcup_{i = 1}^{m}\left\{X_{\left\langle i\right\rangle} \leq x_{\alpha}^{*}\right\}\right) \label{eq:ord-opt-defn} \\
&= \operatorname{Pr}\left( \min_{i \in \left\{1, \dots, m\right\}}X_{\left\langle i\right\rangle} \leq x_{\alpha}^{*}\right),
\end{align}
where $x_{\alpha}^{*}$ with $\alpha \in \left(0, 1\right]$ is such that $\operatorname{Pr}\left(X \leq x_{\alpha}^{*}\right) = \alpha$.
\label{def:ord_opt_general_copula}
\end{definition}

This is an offline variant of the secretary problem, because the selection is made after all the $Z_{1}, \dots, Z_{n}$ are observed. Also, the classical secretary problem is trivial if selections can be made offline, so one purpose of having noisy observations is to `un-trivialise' the problem. Painting this setup in a simulation-based optimisation context, we could for example take each $Z_{i}$ as the empirical average $Z_{i} = \frac{1}{T}\sum_{t=1}^{T}J\left(\theta_{i}, W_{t}\right)$ over $T$ simulation replications, where $J\left(\theta_{i}, W_{t}\right)$ is a system performance function of a randomly generated design parameter $\theta_{i}$, and the $W_{t}$ are stochastic simulation input variables. Naturally, $Z_{i}$ is then a noisy estimate of the actual expectation $X_{i} = \mathbb{E}_{W}\left[J\left(\theta_{i}, W\right)\middle|\theta_{i}\right]$. In essence, simulating a candidate solution is akin to conducting an `interview' for it. Hence Definition \ref{def:ord_opt_general_copula} is compatible, provided the random variables are continuous. \\

A noisy secretary problem was also studied by \cite{Krieger2012}, but instead considered the probability that the best out of $X_{1}, \dots, X_{n}$ was selected. It was demonstrated that, even in the offline setting, there exist distributions whereby this probability tends to zero as $n \to \infty$. In this paper, the role of $\alpha$ is for goal-softening, i.e. controlling degree of sub-optimality as defined by a threshold of the best $100\alpha$ percentile of $X$. The results for our formulation imply that $\alpha$ can be controlled to make $p_{\mathrm{success}}$ arbitrarily close to one, while $p_{\mathrm{success}} > 0$ always when $\alpha \in \left(0, 1\right]$. \\

We also focus attention on the problem of inverting the success probability, i.e. finding a sample size $n$ so that $p_{\mathrm{success}}$ is sufficiently high.
\begin{problem}[High probability inversion]
Prescribe $\delta \in \left(0, 1\right]$, and fix $m \in \mathbb{N}$, $\alpha \in \left(0, 1\right]$. Given the distribution for $\left(Z, X\right)$:
\begin{enumerate}[label=(\alph*)]
\item does there exist a finite $n$ which yields $p_{\mathrm{success}}^{F_{Z,X}}\left(n, m, \alpha\right) \geq 1 - \delta$? \label{prob:inversion_existence}
\item If so, find such an $n$. \label{prob:inversion_n}
\end{enumerate}
\label{prob:inversion}
\end{problem}
In the context of simulation-based optimisation, $n$ represents the quantity of simulations performed offline. Thus Problem \ref{prob:inversion} could be phrased as "find the number of offline simulations, so that at least one candidate actually performs acceptably well, with high probability".

\subsection{Dependence Modelling}

Since $Z$ is a noisily corrupted version of $X$, a reasonable condition that might be desired on $\left(Z, X\right)$ is some notion of positive dependence, i.e. roughly speaking, low (high) $Z$ is predictive of low (high) $X$. One such formal notion of positive dependence is called stochastically increasing positive dependence (SIPD), defined as follows \cite[\S 2.8.2]{Joe2014}.

\begin{definition}[Stochastically increasing positive dependence]
The random variable $X$ is said to be stochastically increasing positive dependent in $Z$ if
\begin{equation}
\operatorname{Pr}\left(X > x\middle|Z = z\right) \leq \operatorname{Pr}\left(X > x\middle|Z = z'\right)
\label{eq:sipd}
\end{equation}
for all $z, z'$ in the support of $Z$, such that $z \leq z'$.
\label{assump:sipd}
\end{definition}

The SIPD condition may be satisfied, for instance, under a particular additive noise causal representation. Additive noise is the same type of noise considered in \cite{Krieger2012} and areas of stochastic simulation such as \cite{Ho1992, Xie2014}. Further elaboration is provided in Appendix \ref{sec:additive_noise}. \\

Due to their utility in modelling dependence structure and flexibility to describe wide classes of distributions, copulas will be convenient for modelling the pair $\left(Z, X\right)$.

\begin{definition}[Bivariate copula]
A bivariate copula is a bivariate distribution, with both marginal distributions being the $\operatorname{Uniform}\left(0, 1\right)$ distribution.
\end{definition}

Through the \textit{probability integral transform}, which says that both $F_{Z}\left(Z\right)$ and $F_{X}\left(X\right)$ are $\operatorname{Uniform}\left(0, 1\right)$ distributed \cite[Theorem 1]{Angus1994}, every multivariate distribution can be represented with just its marginal distributions and a copula. Moreover, Sklar's theorem \cite[Theorem 1.1]{Joe2014} asserts that if the distribution is continuous, then the choice of copula in this representation is unique. If the continuous pair $\left(Z, X\right)$ is represented with copula $\mathcal{C}_{Z, X}$, then we say that $\left(Z, X\right)$ `has' copula $\mathcal{C}_{Z, X}$ and write the joint copula CDF as
\begin{equation}
\mathcal{C}_{Z, X}\left(z, x\right) := \operatorname{Pr}\left(F_{Z}\left(Z\right) \leq z, F_{X}\left(X\right) \leq x\right).
\end{equation}
In addition, the conditional copula CDF of $X$ given $Z$ is denoted by
\begin{equation}
\mathcal{C}_{X|Z}\left(x\middle|z\right) := \operatorname{Pr}\left(F_{X}\left(X\right) \leq x\middle|F_{Z}\left(Z\right) = z\right).
\end{equation}

\section{Properties of Success Probability}
\label{sec:OO_properties}

In this section, we present several properties of the success probability when $\left(Z, X\right)$ has a general copula model. Our first result gives an expression to compute the success probability by evaluating an $m$-dimensional integral.

\begin{theorem}[Integral expression for success probability]
We have
\begin{multline}
p_{\mathrm{success}}^{F_{Z,X}}\left(n,m,\alpha\right) =\dfrac{n!}{\left(n-m\right)!}\int_{0}^{1}\int_{0}^{z_{m}}\dots\int_{0}^{z_{2}}\left[1-\prod_{j=1}^{m}\left(1-\mathcal{C}_{X|Z}\left(\alpha\middle|z_{j}\right)\right)\right] \\
\times \left(1-z_{m}\right)^{n-m}\mathrm{d}z_{1}\dots \mathrm{d}z_{m-1}\mathrm{d}z_{m}.
\label{eq:p_success_m-dimensional_integral}
\end{multline}
\label{thm:p_success_m-dimensional_integral}
\end{theorem}

The proof may be found in Appendix \ref{sec:proofs}. Theorem \ref{thm:p_success_m-dimensional_integral} reveals that when calculating the success probability given $n$, $m$ and $\alpha$, only the copula of $\left(Z, X\right)$ matters. Or alternatively, we can without loss of generality assume that $\left(Z, X\right)$ is a copula distribution, and write $p_{\mathrm{success}}^{F_{Z,X}} = p_{\mathrm{success}}^{\mathcal{C}_{Z,X}}$. Therefore, we see why the "ordinal" term from OO can be a fitting qualifier - the success probability will be invariant to univariate monotonic (i.e. strictly increasing) transformations of $Z$ and $X$, because this does not modify the underlying copula. \\

The $m$-dimensional integral in \eqref{eq:p_success_m-dimensional_integral} is analytically intractable, except for simple special cases exemplified as follows. For the bivariate Clayton copula \cite[\S 4.6.1]{Joe2014} with parameter $\varrho = 1$ (denoted $\operatorname{Clay}\left(\varrho\right)$), and $n = 3$, $m = 1$, $\alpha = 1/2$, we have
\begin{equation}
p_{\mathrm{success}}^{\operatorname{Clay}\left(1\right)}\left(3, 1, \dfrac{1}{2}\right) = 3\int_{0}^{1} \dfrac{\left(1-z\right)^{2}}{\left(z+1\right)^{2}} \mathrm{d} z = 9 - 12\log 2 \approx 0.6822.
\end{equation}
Or using $\operatorname{Clay}\left(2\right)$ instead, we have
\begin{equation}
p_{\mathrm{success}}^{\operatorname{Clay}\left(2\right)}\left(3, 1, \dfrac{1}{2}\right) = 3\int_{0}^{1} \dfrac{\left(1-z\right)^{2}}{\left(3 z^{2}+1\right)^{3 / 2}} \mathrm{d} z = \dfrac{\sinh ^{-1}\left(\sqrt{3}\right)}{\sqrt{3}} \approx 0.7603.
\end{equation}
Note that it is also possible to consider an alternative selection rule - which is to instead select all the candidates below the threshold $z_{m/n}^{*}$, where $z_{m/n}^{*}$ is the $100m/n$ percentile of $Z$. This selection rule gives a randomised selection size equal to $m$ in expectation, and it can be shown that the analogous success probability is the somewhat more `elegant-looking' expression of $1 - \left(1 - C_{Z,X}\left(m/n,\alpha\right)\right)^{n}$. In comparison to a fixed selection size however, this randomised selection size suffers from three issues: 1) if the marginal distribution of $Z$ is not known in practice, then this selection rule is non-operational; 2) a randomised selection size yields a strictly smaller asymptotic success probability; and 3) in applications (e.g. simulation-based optimisation), it may be undesirable to have the selection size be randomised, especially as the selection size is able to be zero with non-zero probability, and has an asymptotic variance of $m$. Further details are provided in Appendix \ref{sec:randomised_selection_size}.\\

Under a SIPD regularity condition, we can also confirm that the horse-race selection rule (i.e. selecting the best $m$ observed) is indeed the optimal choice.

\begin{proposition}[Optimality of horse-race selection]
If $X$ is SIPD in $Z$, then the selection of the first $m$ order statistics maximises the success probability in Definition \ref{def:ord_opt_general_copula}, compared to any other selection of size $m$ from the sample.
\label{thm:optimality_horse_race}
\end{proposition}

The success probability can also be shown to be non-decreasing in each of its arguments.

\begin{proposition}[Monotonicity of success probability]
The success probability $p_{\mathrm{success}}^{\mathcal{C}_{Z,X}}$ satisfies the following monotonicity properties.
\begin{enumerate}[label=(\alph*)]
\item (Monotonicity in $n$) If $X$ is SIPD in $Z$, then for any $n \in \mathbb{N}$ and fixed $\bar{m} \in \left\{1, \dots, n\right\}$ and fixed $\bar{\alpha} \in \left(0, 1\right]$, we have
\begin{equation}
p_{\mathrm{success}}^{\mathcal{C}_{Z,X}}\left(n, \bar{m}, \bar{\alpha}\right) \leq p_{\mathrm{success}}^{\mathcal{C}_{Z,X}}\left(n + 1, \bar{m}, \bar{\alpha}\right).
\end{equation}
\label{thm:monotonicity_OO_n}
\item (Monotonicity in $m$) For fixed $\bar{n} \in \mathbb{N}$, fixed $\bar{\alpha} \in \left(0, 1\right]$ and any $m \in \left\{1, \dots, n-1\right\}$, we have
\begin{equation}
p_{\mathrm{success}}^{\mathcal{C}_{Z,X}}\left(\bar{n}, m, \bar{\alpha}\right) \leq p_{\mathrm{success}}^{\mathcal{C}_{Z,X}}\left(\bar{n}, m + 1, \bar{\alpha}\right).
\end{equation}
\label{thm:monotonicity_OO_m}
\item (Monotonicity in $\alpha$) For fixed $\bar{n} \in \mathbb{N}$, fixed $\bar{m} \in \left\{1, \dots, \bar{n}\right\}$ and any $\alpha, \alpha' \in \left(0, 1\right]$ such that $\alpha \leq \alpha'$, we have
\begin{equation}
p_{\mathrm{success}}^{\mathcal{C}_{Z,X}}\left(\bar{n}, \bar{m}, \alpha\right) \leq p_{\mathrm{success}}^{\mathcal{C}_{Z,X}}\left(\bar{n}, \bar{m}, \alpha'\right).
\end{equation}
\label{thm:monotonicity_OO_alpha}
\end{enumerate}
\label{thm:monotonicity_OO}
\end{proposition}

For any copula satisfying the SIPD condition, we also have the following general upper and lower bounds.

\begin{proposition}[General bounds for success probability]
If $X$ is SIPD in $Z$, the success probability $p_{\mathrm{success}}^{\mathcal{C}_{Z,X}}$ satisfies
\begin{equation}
1 - \left(1 - \alpha\right)^{m} \leq p_{\mathrm{success}}^{\mathcal{C}_{Z,X}}\left(n, m, \alpha\right) \leq 1 - \left(1 - \alpha\right)^{n}.
\label{eq:general_bounds_OO}
\end{equation}
\label{thm:general_bounds_OO}
\end{proposition}

These bounds are tight, as some of the following limiting cases show.

\begin{proposition}[Limiting forms of success probability]
If $X$ is SIPD in $Z$, the success probability $p_{\mathrm{success}}^{\mathcal{C}_{Z,X}}$ satisfies the following.
\begin{enumerate}[label=(\alph*)]
\item When $m = n$:
\begin{equation}
p_{\mathrm{success}}^{\mathcal{C}_{Z,X}}\left(n, n, \alpha\right) = 1 - \left(1 - \alpha\right)^{n}.
\end{equation}
\label{thm:p_success_limiting_m_equals_n}
\item When $\alpha = 1$:
\begin{equation}
p_{\mathrm{success}}^{\mathcal{C}_{Z,X}}\left(n, m, 1\right) = 1.
\end{equation}
\label{thm:p_success_limiting_alpha_1}
\item When $\alpha \to 0^{+}$:
\begin{equation}
\lim_{\alpha \to 0^{+}}p_{\mathrm{success}}^{\mathcal{C}_{Z,X}}\left(n, m, \alpha\right) = 0.
\end{equation}
\label{thm:p_success_limiting_alpha_to_0}
\item If $Z$ and $X$ are comonotonic (i.e. perfect positive dependence, such that $F_{Z}\left(Z\right) = F_{X}\left(X\right)$ for every realisation of $\left(Z, X\right)$), then
\begin{equation}
p_{\mathrm{success}}^{\mathcal{C}_{Z,X}}\left(n, m, \alpha\right) = 1 - \left(1 - \alpha\right)^{n}.
\end{equation}
\label{thm:p_success_limiting_comonotonic}
\item If $Z$ and $X$ are independent, then
\begin{equation}
p_{\mathrm{success}}^{\mathcal{C}_{Z,X}}\left(n, m, \alpha\right) = 1 - \left(1 - \alpha\right)^{m}.
\end{equation}
\label{thm:p_success_limiting_independent}
\item When $m \to \infty$ and $n \to \infty$ in any way such that $m \leq n$:
\begin{equation}
\lim_{m,n\to\infty}p_{\mathrm{success}}^{\mathcal{C}_{Z,X}}\left(n, m, \alpha\right) = 1.
\end{equation}
\label{thm:p_success_limiting_m_to_infty}
\item When $m$ is finite and $n \to \infty$:
\begin{equation}
\lim_{n\to\infty}p_{\mathrm{success}}^{\mathcal{C}_{Z,X}}\left(n, m, \alpha\right) = 1 - \left(1 - \mathcal{C}_{X|Z}\left(\alpha\middle|0\right)\right)^{m}.
\end{equation}
\label{thm:p_success_limiting_n_to_infty}
\end{enumerate}
\label{thm:p_success_limiting}
\end{proposition}
Due to the specification of Problem \ref{prob:inversion}, we are primarily interested in the regime of large $n$ and relatively small fixed $m$. Sadly, the property in Proposition \ref{thm:p_success_limiting}\ref{thm:p_success_limiting_n_to_infty} indicates that in general, for finite $m$ we have $\lim_{n\to\infty}p_{\mathrm{success}}^{\mathcal{C}_{Z,X}}\left(n, m, \alpha\right) < 1$, whenever $\mathcal{C}_{X|Z}\left(\alpha\middle|0\right) < 1$. For example, the bivariate Frank copula \cite[\S 4.5.1]{Joe2014} with parameter $\varpi > 0$ has boundary conditional CDF
\begin{equation}
\mathcal{C}_{X|Z}^{\mathrm{Frank}}\left(\alpha\middle|0;\varpi\right)=\dfrac{1-e^{-\varpi \alpha}}{1-e^{-\varpi}} < 1.
\end{equation}
Thus to address Problem \ref{prob:inversion}\ref{prob:inversion_existence} in general, we can only attain a high success probability with a combination of sufficiently large $n$ and $m$. However, there do exist classes of copulas where $\lim_{n\to\infty}p_{\mathrm{success}}\left(n, m, \alpha\right) = 1$ for all finite $m \geq 1$. In these cases, it is possible to address Problem \ref{prob:inversion}\ref{prob:inversion_n}, i.e. for a prescribed high probability $1 - \delta$ (with any $\delta \in \left(0, 1\right]$ and $m$, $\alpha$ fixed), to find an $n^{*}$ such that $p_{\mathrm{success}}^{\mathcal{C}_{Z,X}}\left(n^{*}, m, \alpha\right) \geq 1 - \delta$. A naive approach would be to increment $n$ and numerically evaluate $p_{\mathrm{success}}^{\mathcal{C}_{Z,X}}$ using the $m$-dimensional integral \eqref{eq:p_success_m-dimensional_integral} until such an $n^{*}$ is found. In the next section, we study one such class of copula which satisfies this property, and explore an alternative approach for finding $n^{*}$ using an analytic lower bound for $p_{\mathrm{success}}$.

\section{Gaussian Copula Success Probability}
\label{sec:gaussian_copula_ordinal_optimisation}

The naive approach to address Problem \ref{prob:inversion}\ref{prob:inversion_n} is by numerically evaluating the $m$-dimensional integral iteratively (discussed at the end of Section \ref{sec:OO_properties}). However, several numerical issues pervade this approach.
\begin{itemize}
\item Firstly, $m$-dimensional integrals will become hard to compute (even numerically) for sufficiently large $m$. 
\item Secondly, even by taking $m = 1$ (which lower bounds the success probability for $m \geq 1$ due to Proposition \ref{thm:monotonicity_OO}\ref{thm:monotonicity_OO_m}), we do not know the vicinity of the order of magnitude for $n$ to begin evaluating $p_{\mathrm{success}}^{\mathcal{C}_{Z,X}}$. For example, we could begin evaluating $p_{\mathrm{success}}^{\mathcal{C}_{Z,X}}$ by incrementing $n$ starting from $n = 100$, but it would actually require $n \approx 10^{9}$ before $p_{\mathrm{success}}^{\mathcal{C}_{Z,X}} \geq 1 - \delta$, resulting in an excessive number of evaluations.
\item Lastly, even if we did know (or can narrow down) the vicinity of the order of magnitude for $n$, it may not be possible to obtain an accurate numerical evaluation of $p_{\mathrm{success}}^{\mathcal{C}_{Z,X}}$ for very large $n$, since the factor $\left(1 - z_{m}\right)^{n - m}n!/\left(n - m\right)!$ related to the probability density in \eqref{eq:p_success_m-dimensional_integral} tends towards a degenerate distribution concentrated over zero as $n \to \infty$, for fixed $m$.
\end{itemize}
In this section, we address the above issues by specialising the success probability to the case when $\left(Z, X\right)$ has a bivariate Gaussian copula. This allows us to develop an analytic lower bound for the success probability, which may be inverted to directly address Problem \ref{prob:inversion}\ref{prob:inversion_n}. The bivariate Gaussian copula is defined as follows.

\begin{definition}[Bivariate Gaussian copula]
Let $\left(Z, X\right)$ be a bivariate standard Gaussian vector with correlation $\rho \in \left[-1, 1\right]$, i.e.
\begin{equation}
\begin{bmatrix}Z\\
X
\end{bmatrix}\sim\mathcal{N}\left(\begin{bmatrix}0\\
0
\end{bmatrix},\begin{bmatrix}1 & \rho\\
\rho & 1
\end{bmatrix}\right).
\label{eq:bivariate-gaussian-copula}
\end{equation}
Then the bivariate Gaussian copula with correlation $\rho$, denoted $\mathcal{NC}\left(\rho\right)$, is the distribution of $\left(\Phi\left(Z\right), \Phi\left(X\right)\right)$.
\label{def:bivariate_gaussian_copula}
\end{definition}
The class of multivariate distributions with Gaussian copulas are referred to as \textit{non-paranormal} distributions \cite{Liu2012a}, and alternatively as \textit{meta-Gaussian} distributions \cite{Storvik2009}. A bivariate Gaussian copula is entirely specified by the correlation parameter $\rho$, which also neatly summarises the dependence within the distribution. Thus, we denote the success probability with a Gaussian copula as $p_{\mathrm{success}}^{\mathcal{NC}\left(\rho\right)}\left(n, m, \alpha\right)$. The bivariate Gaussian copula is known to satisfy the stronger condition of \textit{positive likelihood ratio dependence} for $\rho > 0$ \cite[\S 4.3.1]{Joe2014}, which implies SIPD of $X$ in $Z$ \cite[Theorem 5.2.19]{Nelsen1999}. Thus all the properties in Section \ref{sec:OO_properties} hold for $p_{\mathrm{success}}^{\mathcal{NC}\left(\rho\right)}$. It is also known that the boundary CDF for the bivariate Gausssian is $\mathcal{C}_{X|Z}^{\mathcal{N}}\left(\cdot\middle| 0; \rho\right) = 1$ for all $\rho > 0$ \cite[\S 4.3.1]{Joe2014}, so by Proposition \ref{thm:p_success_limiting}\ref{thm:p_success_limiting_n_to_infty}, we immediately have
\begin{equation}
\lim_{n\to\infty}p_{\mathrm{success}}^{\mathcal{NC}\left(\rho\right)}\left(n, m, \alpha\right) = 1
\label{eq:convergence_n_gaussian_copula}
\end{equation}
when $\rho > 0$, for any $m \geq 1$. The Gaussian copula model for $\left(Z, X\right)$ also admits an additive independent Gaussian noise representation (see Appendix \ref{sec:additive_noise}).    

\subsection{Analytic Lower Bound for Success Probability}
\label{sec:analytic_lower_bound}

Under a Gaussian copula model for $\left(Z, X\right)$ with positive correlation $\rho > 0$, we can derive the following analytic lower bound on the success probability.
\begin{theorem}[Success probability analytic lower bound]
Given some $\omega \in \left(0, \frac{\pi}{2}\right)$, let
\begin{equation}
\mathfrak{c}_{1} = \dfrac{1}{2} - \dfrac{\omega}{\pi}, \quad \mathfrak{c}_{2} = \dfrac{\cot{\omega}}{\pi - 2\omega}
\end{equation}
and
\begin{gather}
\mu_{n} = -\sqrt{\dfrac{\log\left(n\mathfrak{c}_{1}\right)}{\mathfrak{c}_{2}}} \label{eq:mu-constructed} \\
\sigma_{n}^{2} = \dfrac{-\log\log 2}{2\mathfrak{c}_{2}\left(\log\left(n\mathfrak{c}_{1}\right) - \log\log 2\right)}. \label{eq:sigma-constructed}
\end{gather}
Then for any $\omega \in \left(0, \frac{\pi}{2}\right)$, there exists an $\widetilde{n}\left(\omega\right) \in \mathbb{N}$ such that for all $n \geq \widetilde{n}\left(\omega\right)$, $m \in \left\{1, \dots, n\right\}$, $\rho \in \left(0, 1\right]$, and $\alpha \in \left(0, 1\right]$, we have
\begin{equation}
p_{\mathrm{success}}^{\mathcal{NC}\left(\rho\right)}\left(n, m, \alpha\right) \geq p_{\mathrm{success}}^{\mathcal{NC}\left(\rho\right)}\left(n, 1, \alpha\right) \geq \Phi\left(\dfrac{\Phi^{-1}\left(\alpha\right) - \rho\mu_{n}}{\sqrt{1 - \rho^{2} + \rho^{2}\sigma_{n}^{2}}}\right).
\label{eq:success-prob-copula-lb}
\end{equation}
\label{thm:constructed-lower-bound}
\end{theorem}

The proof of this result is found in Appendix \ref{sec:proofs}. Furthermore, given any $n \in \mathbb{N}$, $m \in \left\{1, \dots, n\right\}$, $\rho \in \left(0, 1\right]$, and $\alpha \in \left(0, 1\right]$, we can optimise the bound with respect to $\omega$ by
\begin{equation}
p_{\mathrm{success}}^{\mathcal{NC}\left(\rho\right)}\left(n, m, \alpha\right) \geq p_{\mathrm{success}}^{\mathcal{NC}\left(\rho\right)}\left(n, 1, \alpha\right) \geq \sup_{\omega \in \Omega_{n}}\Phi\left(\dfrac{\Phi^{-1}\left(\alpha\right) - \rho\mu_{n}\left(\omega\right)}{\sqrt{1 - \rho^{2} + \rho^{2}\left[\sigma_{n}\left(\omega\right)\right]^{2}}}\right),
\label{eq:success-prob-copula-lb-optimised} 
\end{equation}
where $\Omega_{n} \subset \left(0, \frac{\pi}{2}\right)$ is the set of all $\omega$ such that $n \geq \widetilde{n}\left(\omega\right)$, while $\mu_{n}\left(\omega\right)$, $\left[\sigma_{n}\left(\omega\right)\right]^{2}$ are \eqref{eq:mu-constructed}, \eqref{eq:sigma-constructed} respectively but with dependence on $\omega$ explicitly denoted. For a given $\omega$, it is also worthwhile to consider the smallest integer $\widetilde{n}\left(\omega\right)$ such that \eqref{eq:success-prob-copula-lb} is valid. It is clear that we must have $n^{*}\left(\omega\right) > 1/\mathfrak{c}_{1}$, otherwise it possibly allows for $\log\left(n\mathfrak{c}_{1}\right) < 0$ in \eqref{eq:mu-constructed} and \eqref{eq:sigma-constructed}. Given $n$ and $\omega$, one can numerically certify whether $n \geq \widetilde{n}\left(\omega\right)$, using sufficient conditions from the proof of Theorem \ref{thm:constructed-lower-bound}. We have empirically observed that $\widetilde{n}\left(\omega\right)$ can be quite small; we are usually able to accept $\widetilde{n}\left(\omega\right) = \left\lceil 1/\mathfrak{c}_{1}\right\rceil$. Using this certification, the optimised lower bound \eqref{eq:success-prob-copula-lb-optimised} can also be implemented via a numerical procedure, noting that we need only conduct search over a univariate bounded interval. Further discussion and pseudocode for these implementations can be found in Appendix \ref{sec:algorithms}. \\

Throughout Figures \ref{fig:approx_n2_v3}-\ref{fig:approx_alpha_v3}, we plot the optimised lower bound \eqref{eq:success-prob-copula-lb-optimised} from Theorem \ref{thm:constructed-lower-bound}, with baseline values $m = 1$, $\alpha = 0.05$, $\rho = 0.4$ and $n = 100$, while varying s single quantity. In Figure \ref{fig:approx_n2_v3}, this is compared against a numerical evaluation of the success probability using \eqref{eq:p_success_m-dimensional_integral}. However, since the density in the integrand concentrates over zero as $n$ increases, numerical integration becomes inaccurate for large $n$ (as discussed at the beginning of this section). In Figure \ref{fig:approx_n_v3}, we instead plot the lower bound over a semi-log horizontal axis scale for large $n$, to illustrate the convergence of the success probability to one. These plots demonstrate that the behaviour of the lower bound is reasonably close to the actual probability. As the lower bound has been derived with $m = 1$ while the bound itself does not change with $m$, this means the bound is least conservative for $m = 1$, and will generally become more conservative as $m$ grows. For instance with $\alpha = 0.05$ and $m = 32$, the lower bound from Proposition \ref{thm:general_bounds_OO} yields $p_{\mathrm{success}}^{\mathcal{N}} \geq 0.806$ for any $n \geq 32$, already surpassing the lower bound with $n = 10^{9}$ from Figure \ref{fig:approx_n2_v3}. However, the bound \eqref{eq:success-prob-copula-lb} is still useful in the regime of small $m$ and large $n$, e.g. for addressing Problem \ref{prob:inversion}\ref{prob:inversion_n}.
    
    \begin{figure}[htb!]
        \centering
        \vskip\baselineskip
        \begin{subfigure}[t]{0.49\textwidth}
            \centering
            \includegraphics[width=\textwidth]{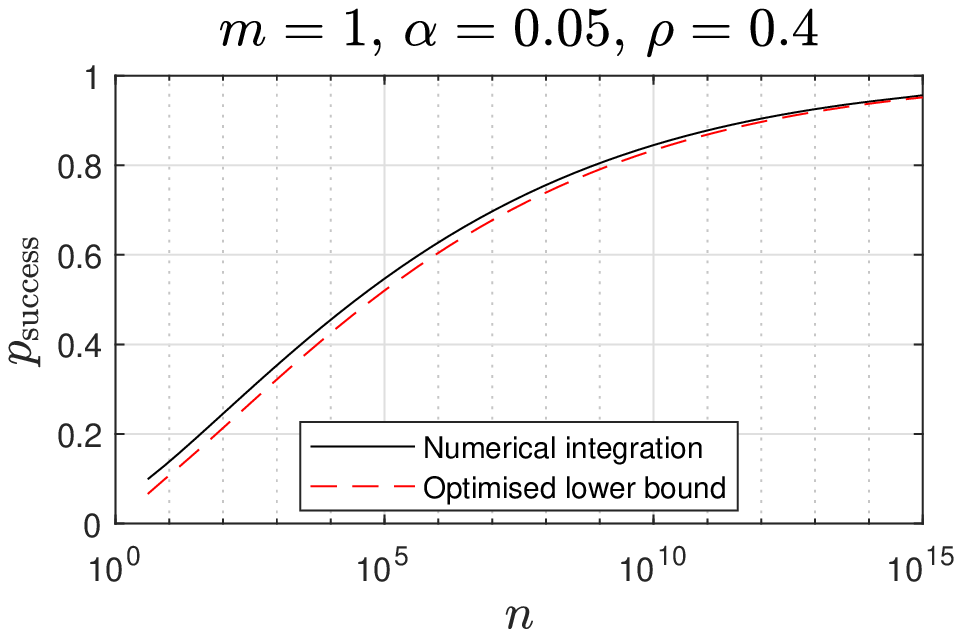}
            \caption{Comparison of the optimised lower bound in \eqref{eq:success-prob-copula-lb-optimised} to numerical integration via Theorem \ref{thm:p_success_m-dimensional_integral}, as $n$ is varied. In addition, this figure demonstrates monotonicity in $n$ from Proposition \ref{thm:monotonicity_OO}\ref{thm:monotonicity_OO_n}.}
            \label{fig:approx_n2_v3}
        \end{subfigure}
        \hfill
        \begin{subfigure}[t]{0.49\textwidth}  
            \centering 
            \includegraphics[width=\textwidth]{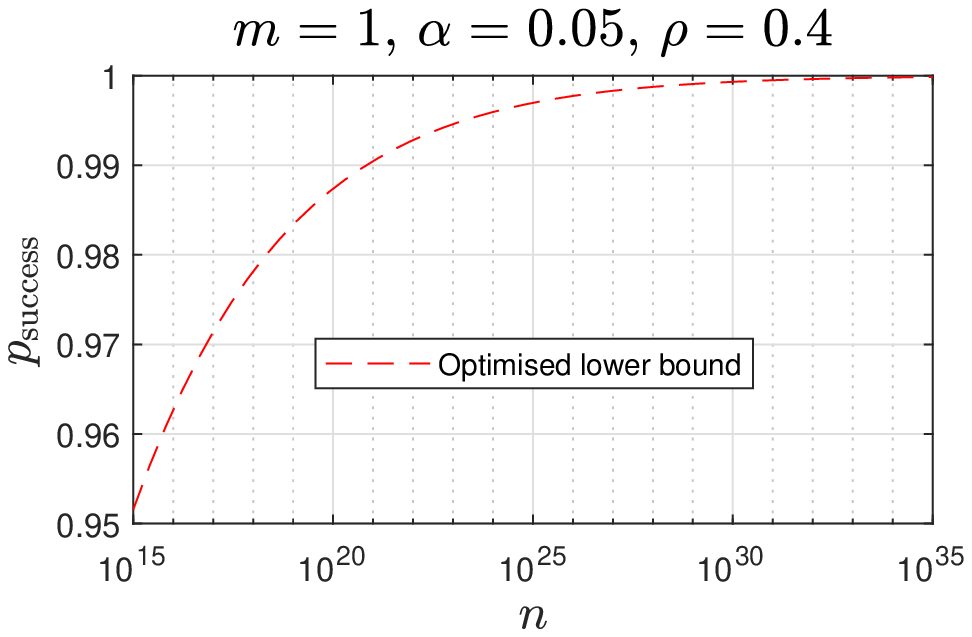}
            \caption{Comparison of the lower bound \eqref{eq:success-prob-copula-lb-optimised} when $n$ is varied over a semi-log horizontal axis scale. In addition, this figure demonstrates convergence to one in $n$ from \eqref{eq:convergence_n_gaussian_copula}.}

            \label{fig:approx_n_v3}
        \end{subfigure}
        \vskip\baselineskip
        \begin{subfigure}[t]{0.49\textwidth}   
            \centering 
            \includegraphics[width=\textwidth]{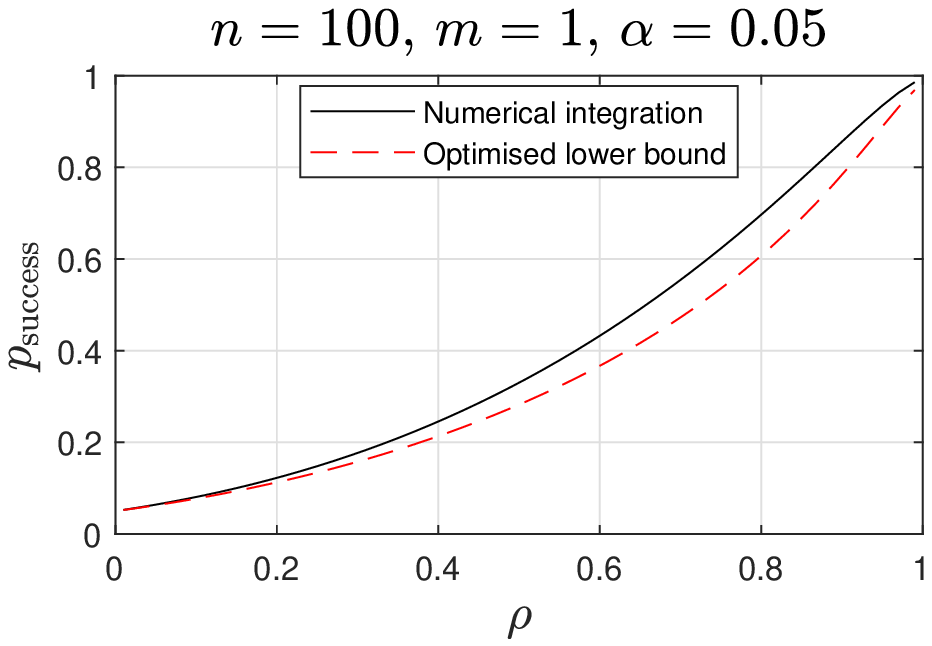}
            \caption{Comparison of the optimised lower bound in \eqref{eq:success-prob-copula-lb-optimised} to numerical integration via Theorem \ref{thm:p_success_m-dimensional_integral}, as $\rho$ is varied.}
            \label{fig:approx_rho_v3}
        \end{subfigure}
        \hfill
        \begin{subfigure}[t]{0.49\textwidth}   
            \centering 
            \includegraphics[width=\textwidth]{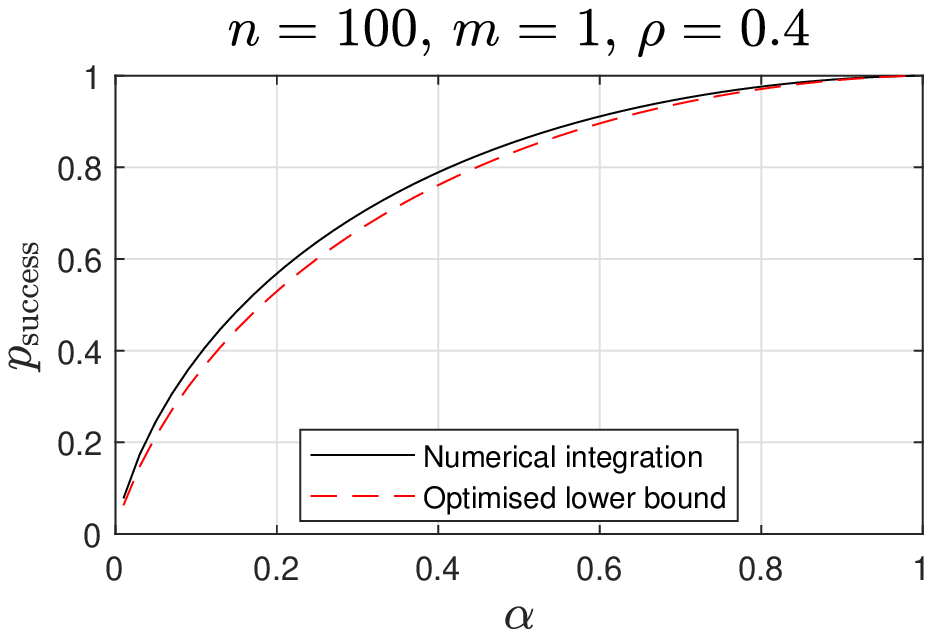}
            \caption{Comparison of the optimised lower bound in \eqref{eq:success-prob-copula-lb-optimised} to numerical integration via Theorem \ref{thm:p_success_m-dimensional_integral}, as $\alpha$ is varied. In addition, this figure demonstrates monotonicity in $\alpha$ from Proposition \ref{thm:monotonicity_OO}\ref{thm:monotonicity_OO_alpha}.}
            \label{fig:approx_alpha_v3}
        \end{subfigure}
        
               \caption{\small Numerical results for the lower bound.} 
        \label{fig:numerical2}
    \end{figure}

\subsection{Inversion of Analytic Lower Bound}
\label{sec:invert_lower_bound}

We now use the analytic lower bound in \eqref{eq:success-prob-copula-lb} of Theorem \ref{thm:constructed-lower-bound} to address Problem \ref{prob:inversion}\ref{prob:inversion_n}, in the case that $\left(Z, X\right)$ has a Gaussian copula. Putting this into the contexts discussed earlier, we may address questions of the nature "how many job candidates should the hiring manager interview", or "how many simulations should be conducted offline" in order to guarantee a prescribed high success probability. Observe that the analytic lower bound in \eqref{eq:success-prob-copula-lb} tends to one, as $n \to \infty$. So to address Problem \ref{prob:inversion}\ref{prob:inversion_n} using the lower bound (given $\alpha$, $\rho$ and $\delta$), we aim to invert for $n$ in terms of $\omega$ with the expression
\begin{equation}
\Phi\left(\dfrac{\Phi^{-1}\left(\alpha\right) - \rho\mu_{n}}{\sqrt{1 - \rho^{2} + \rho^{2}\sigma_{n}^{2}}}\right) = 1- \delta.
\end{equation}
Putting the definitions of $\mu_{n}$ and $\sigma_{n}$ from \eqref{eq:mu-constructed} and \eqref{eq:sigma-constructed} respectively, this equation can be rearranged into a quartic equation in $\mathsf{x} = \sqrt{\log\left(n\mathfrak{c}_{1}\right)}$, of the form
\begin{equation}
\mathfrak{a}_{4}\mathsf{x}^{4} + \mathfrak{a}_{3}\mathsf{x}^{3} + \mathfrak{a}_{2}\mathsf{x}^{2} + \mathfrak{a}_{1}\mathsf{x} + \mathfrak{a}_{0} = 0,
\label{eq:quartic}
\end{equation}
where
\begin{align}
\mathfrak{a}_{4} &= -\dfrac{2\rho^{2}}{\log\log 2} \\
\mathfrak{a}_{3} &= -\dfrac{4\Phi^{-1}\left(\alpha\right)\rho\sqrt{\mathfrak{c}_{2}}}{\log\log2} \\
\mathfrak{a}_{2} &= 2\rho^{2}-\dfrac{2\mathfrak{c}_{2}\left(\left[\Phi^{-1}\left(\alpha\right)\right]^{2}-\left[\Phi^{-1}\left(1-\delta\right)\right]^{2}+\rho^{2}\left[\Phi^{-1}\left(1-\delta\right)\right]^{2}\right)}{\log\log2} \\
\mathfrak{a}_{1} &= 4\sqrt{\mathfrak{c}_{2}}\Phi^{-1}\left(\alpha\right)\rho \\
\mathfrak{a}_{0} &= 2\mathfrak{c}_{2}\left(\left[\Phi^{-1}\left(\alpha\right)\right]^{2}-\left[\Phi^{-1}\left(1-\delta\right)\right]^{2}+\rho^{2}\left[\Phi^{-1}\left(1-\delta\right)\right]^{2}\right)-\rho^{2}\left[\Phi^{-1}\left(1-\delta\right)\right]^{2}.
\end{align}
Therefore we take the solution for $n$ corresponding to $\mathsf{x} = \sqrt{\log\left(n\mathfrak{c}_{1}\right)}$ as the greatest real root of the quartic equation. Let this solution for $n$ in terms of $\omega$ be denoted $n^{*}\left(\omega\right)$. According to the monotonicity and convergence properties from Proposition \ref{thm:monotonicity_OO}\ref{thm:monotonicity_OO_n} and \eqref{eq:convergence_n_gaussian_copula} respectively, then provided $n^{*}\left(\omega\right) \geq \widetilde{n}\left(\omega\right)$, we guarantee
\begin{equation}
p_{\mathrm{success}}^{\mathcal{NC}\left(\rho\right)}\left(n^{*}\left(\omega\right), m, \alpha\right) \geq 1 - \delta,
\end{equation}
since $n^{*}\left(\omega\right)$ upper bounds the smallest $n$ needed such that $p_{\mathrm{success}} \geq 1 - \delta$. Moreover, in a similar way to \eqref{eq:success-prob-copula-lb-optimised}, one can numerically optimise with respect to $\omega$ to find the smallest $n^{*}\left(\omega\right)$ that guarantees a high probability of success. Pseudocode implementing this (which also takes into account the requirement $n^{*}\left(\omega\right) \geq \widetilde{n}\left(\omega\right)$) can be found in Algorithm \ref{alg:optimimised_n_high_prob} of Appendix \ref{sec:algorithms}.

\begin{table}[hbt!]
\centering
\caption{Computed values of $n$ which guarantees $p_{\mathrm{success}}^{\mathcal{NC}\left(\rho\right)}\left(n^{*}\left(\omega\right), m, \alpha\right) \geq 1 - \delta$, with fixed $\alpha = 0.01$ and valid for any $m \geq 1$.}
\label{tab:high-prob}
\begin{tabular}{|l|ccc|}
\hline
   & $\delta = 0 .01$        & $\delta = 0.05$         &  $\delta = 0.1$ \\ \hline
$\rho = 0.01$ & $8.144\times 10^{47007}$       & $5.427\times 10^{34246}$ & $8.943 \times 10^{28267}$ \\ 
$\rho = 0.3$ & $3.289\times 10^{51}$       & $1.619\times 10^{38}$ & $8.775 \times 10^{31}$ \\ 
$\rho = 0.6$  & $8.703 \times 10^{11}$ & $1.988 \times 10^{9}$       & $1.078 \times 10^{8}$ \\ 
$\rho = 0.9$  & $16744$     & $4338$     & $2188$   \\ 
$\rho = 0.99$  & $893$     & $505$     & $372$   \\ \hline
\end{tabular}
\end{table}

Table \ref{tab:high-prob} lists computed values of $n^{*}$ numerically optimised with respect to $\omega$, using the aforementioned approach. The table is valid for all $m \geq 1$ (however are values are least conservative when $m = 1$), for fixed $\alpha = 0.01$ and a variety of values for $\rho$ and $\delta$. The values for $n$ trend downwards as $\rho$ increases, which is intuitive (as fewer samples might be required if noisy observations are strongly correlated with the actual values). Of particular note, the case with extremely small correlation $\rho = 0.01$ requires $n$ to be at an impractical order of magnitude, namely $10^{47007}$ when $\delta = 0.01$. This highlights the utility of the analytic lower bound in Theorem \ref{thm:constructed-lower-bound}, which allows for a computationally cheap procedure to find a sufficiently high $n$. If Problem \ref{prob:inversion}\ref{prob:inversion_n} were attempted to be solved by evaluating expression \eqref{eq:p_success_m-dimensional_integral}, then large $n$ such as in the order of $10^{47007}$ would have rendered the evaluation of such probabilities to be intractable. \\

Also, Table \ref{tab:high-prob} illustrates the value of having a strong positive dependence in $\rho$, since it reduces the sample size required to reach a prescribed high probability of success. For instance with $\delta = 0.1$, increasing from $\rho = 0.01$ to $\rho = 0.6$ reduces the order of magnitude required for $n^{*}$ from $10^{28267}$ to a more practical $10^{8}$. As one may have reasonably guessed, increasing the strength of correlation between $Z$ and $X$ (which is in effect, decreasing the amount of noise in the interview/simulation of candidates) has a favourable effect on the success probability.

\section{Conclusion}
\label{sec:conclusion}

Motivated by methods in simulation-based optimisation, we studied the success probability for an offline multiple noisy secretary problem. As a consequence of applying an ordinal optimisation selection rule with a notion of goal softening, the success probability was found to depend only on the underlying copula. One key condition for the copula was that of stochastically increasing positive dependence, which is sufficient to ensure that $p_{\mathrm{success}}$ is non-decreasing in $n$. An analytic lower bound for $p_{\mathrm{success}}$ was developed in the case of the Gaussian copula model, and Figures \ref{fig:approx_n2_v3}-\ref{fig:approx_alpha_v3} illustrate that this lower bound is close to the true success probability. The lower bound may also be inverted to compute sufficiently large values of $n$ which guarantees success probabilities arbitrarily close to one. This is demonstrated in numerical examples by successfully finding such values of $n$, even in regimes which require extremely high orders of magnitude for $n$, whereby numerical integration would not be successful. \\

The following directions are proposed for future work. In order to address Problem \ref{prob:inversion}\ref{prob:inversion_n} for a non-Gaussian copula, invertible analytic lower bounds for other classes of copulae could be investigated. It may also be interesting to study how strength of dependence affects $p_{\mathrm{success}}$ for wider classes of copulas (analogous to the role of $\rho$ in the Gaussian copula), and how the success probability could be estimated, whenever the copula is not specified exactly in practice. Another direction would be to consider a relaxed problem where $\left(Z, X\right)$ is allowed to have a discrete distribution. The present approach may then need to be modified, since the underlying copula for $\left(Z, X\right)$ would no longer be unique.

\bibliographystyle{ieeetr}
\bibliography{oo_secretary_preprint}  

\begin{thebibliography}{10}

\bibitem{Dudewicz1980}
E.~J. Dudewicz, ``Ranking (ordering) and selection: An overview of how to
  select the best,'' {\em Technometrics}, vol.~22, no.~1, p.~113, 1980.

\bibitem{Peng2018}
Y.~Peng, E.~K.~P. Chong, C.-H. Chen, and M.~C. Fu, ``Ranking and selection as
  stochastic control,'' {\em {IEEE} Transactions on Automatic Control},
  vol.~63, no.~8, pp.~2359--2373, 2018.

\bibitem{Bubeck2009}
S.~Bubeck, R.~Munos, and G.~Stoltz, ``Pure exploration in multi-armed bandits
  problems,'' in {\em International Conference on Algorithmic Learning Theory},
  pp.~23--37, Springer Berlin Heidelberg, 2009.

\bibitem{Fu2014}
M.~C. Fu, ed., {\em Handbook of Simulation Optimization}.
\newblock Springer, 2014.

\bibitem{Shen2020}
C.~Shen, Z.~Wang, S.~Villar, and M.~V.~D. Schaar, ``Learning for dose
  allocation in adaptive clinical trials with safety constraints,'' in {\em
  International Conference on Machine Learning}, 2020.

\bibitem{Bruss1984}
F.~T. Bruss, ``A unified approach to a class of best choice problems with an
  unknown number of options,'' {\em The Annals of Probability}, vol.~12, no.~3,
  1984.

\bibitem{Gilbert1966}
J.~P. Gilbert and F.~Mosteller, ``Recognizing the maximum of a sequence,'' {\em
  Journal of the American Statistical Association}, vol.~61, no.~313,
  pp.~35--73, 1966.

\bibitem{Preater1994}
J.~Preater, ``On multiple choice secretary problems,'' {\em Mathematics of
  Operations Research}, vol.~19, no.~3, pp.~597--602, 1994.

\bibitem{Krieger2012}
A.~M. Krieger and E.~Samuel-Cahn, ``The noisy secretary problem and some
  results on extreme concomitant variables,'' {\em Journal of Applied
  Probability}, vol.~49, no.~3, pp.~821--837, 2012.

\bibitem{Ho1992}
Y.~C. Ho, R.~S. Sreenivas, and P.~Vakili, ``Ordinal optimization of {DEDS},''
  {\em Discrete Event Dynamic Systems}, vol.~2, no.~1, pp.~61--88, 1992.

\bibitem{Ho2007}
Y.-C. Ho, Q.-C. Zhao, and Q.-S. Jia, {\em Ordinal Optimization: Soft
  Optimization for Hard Problems}.
\newblock Springer, 2007.

\bibitem{Yang2002}
M.~Yang and L.~Lee, ``Ordinal optimization with subset selection rule,'' {\em
  Journal of Optimization Theory and Applications}, vol.~113, no.~3,
  pp.~597--620, 2002.

\bibitem{Vidyasagar2001}
M.~Vidyasagar, ``Randomized algorithms for robust controller synthesis using
  statistical learning theory,'' {\em Automatica}, vol.~37, no.~10,
  pp.~1515--1528, 2001.

\bibitem{Chen2000}
H.-C. Chen, C.-H. Chen, and E.~Yucesan, ``Computing efforts allocation for
  ordinal optimization and discrete event simulation,'' {\em {IEEE}
  Transactions on Automatic Control}, vol.~45, no.~5, pp.~960--964, 2000.

\bibitem{Lattimore2020}
T.~Lattimore and C.~Szepesv{\'{a}}ri, {\em Bandit Algorithms}.
\newblock Cambridge University Press, 2020.

\bibitem{Peng2018a}
X.~B. Peng, M.~Andrychowicz, W.~Zaremba, and P.~Abbeel, ``Sim-to-real transfer
  of robotic control with dynamics randomization,'' in {\em {IEEE}
  International Conference on Robotics and Automation ({ICRA})}, {IEEE}, 2018.

\bibitem{Biller2009}
B.~Biller, ``Copula-based multivariate input models for stochastic
  simulation,'' {\em Operations Research}, vol.~57, no.~4, pp.~878--892, 2009.

\bibitem{Shaked2007}
M.~Shaked and J.~G. Shanthikumar, {\em Stochastic Orders}.
\newblock Springer, 2007.

\bibitem{Joe2014}
H.~Joe, {\em Dependence Modeling with Copulas}.
\newblock CRC Press, 2014.

\bibitem{Xie2014}
W.~Xie, B.~L. Nelson, and R.~R. Barton, ``A bayesian framework for quantifying
  uncertainty in stochastic simulation,'' {\em Operations Research}, vol.~62,
  no.~6, pp.~1439--1452, 2014.

\bibitem{Angus1994}
J.~E. Angus, ``The probability integral transform and related results,'' {\em
  {SIAM} Review}, vol.~36, no.~4, pp.~652--654, 1994.

\bibitem{Liu2012a}
H.~Liu, F.~Han, M.~Yuan, J.~Lafferty, and L.~Wasserman, ``High-dimensional
  semiparametric gaussian copula graphical models,'' {\em The Annals of
  Statistics}, vol.~40, no.~4, pp.~2293--2326, 2012.

\bibitem{Storvik2009}
B.~Storvik, G.~Storvik, and R.~Fjortoft, ``On the combination of multisensor
  data using meta-gaussian distributions,'' {\em {IEEE} Transactions on
  Geoscience and Remote Sensing}, vol.~47, no.~7, pp.~2372--2379, 2009.

\bibitem{Nelsen1999}
R.~B. Nelsen, {\em An Introduction to Copulas}.
\newblock Springer, 1st~ed., 1999.

\bibitem{Morgan2015}
S.~L. Morgan and C.~Winship, {\em Counterfactuals and Causal Inference: Methods
  and Principles for Social Research}.
\newblock Cambridge University Press, 2nd~ed., 2015.

\bibitem{Harder2016}
M.~Harder, {\em Exchangeability of Copulas}.
\newblock PhD thesis, Ulm University, 2016.

\bibitem{David2005}
H.~A. David and H.~N. Nagaraja, {\em Order Statistics}.
\newblock John Wiley, 2005.

\bibitem{Wu2018}
M.-W. Wu, Y.~Li, M.~Gurusamy, and P.-Y. Kam, ``A tight lower bound on the
  gaussian {Q}-function with a simple inversion algorithm, and an application
  to coherent optical communications,'' {\em {IEEE} Communications Letters},
  vol.~22, no.~7, pp.~1358--1361, 2018.

\bibitem{Chiani2003}
M.~Chiani, D.~Dardari, and M.~Simon, ``New exponential bounds and
  approximations for the computation of error probability in fading channels,''
  {\em {IEEE} Transactions on Wireless Communications}, vol.~24, no.~5,
  pp.~840--845, 2003.

\bibitem{Arnold2008}
B.~C. Arnold, N.~Balakrishnan, and H.~N. Nagaraja, {\em A First Course in Order
  Statistics}.
\newblock SIAM, 2008.

\bibitem{Rasmussen2006}
C.~E. Rasmussen and C.~K.~I. Williams, {\em Gaussian Processes for Machine
  Learning}.
\newblock MIT Press, 2006.

\bibitem{Krishnamurthy2016}
V.~Krishnamurthy, {\em Partially Observed Markov Decision Processes: From
  Filtering to Controlled Sensing}.
\newblock Cambridge University Press, 2016.

\end{thebibliography}






\begin{appendices}

\section{Additive Noise Representations}
\label{sec:additive_noise}

Consider the following causal mechanism (depicted in Figure \ref{fig:causal_X}) for generating $\left(Z, X\right)$: by first generating $Z$, and then generating $X$ given $Z = z$ though
\begin{equation}
X = z + Y,
\label{eq:additive_noise_X}
\end{equation}
where $Y$ is independent of $Z$. An alternative additive causal representation is the reverse (depicted in Figure \ref{fig:causal_Z}): first $X$ is generated, and then $Z$ is generated given $X = x$ through
\begin{equation}
Z = x + Y,
\label{eq:additive_noise_Z}
\end{equation}
where $Y$ is independent of $X$.

 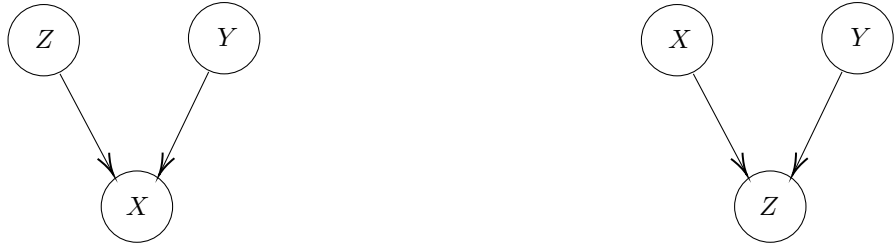
\begin{figure}[hbt!]
       \centering
       \vskip\baselineskip
        \begin{subfigure}[t]{0.49\textwidth}
        \vskip 0pt
        \begin{tikzpicture}[x=0.75pt,y=0.75pt,yscale=-1,xscale=1]

\draw   (314.8,148.2) .. controls (314.8,138.25) and (322.86,130.2) .. (332.8,130.2) .. controls (342.75,130.2) and (350.8,138.25) .. (350.8,148.2) .. controls (350.8,158.14) and (342.75,166.2) .. (332.8,166.2) .. controls (322.86,166.2) and (314.8,158.14) .. (314.8,148.2) -- cycle ;
\draw   (267.8,64.2) .. controls (267.8,54.25) and (275.86,46.2) .. (285.8,46.2) .. controls (295.75,46.2) and (303.8,54.25) .. (303.8,64.2) .. controls (303.8,74.14) and (295.75,82.2) .. (285.8,82.2) .. controls (275.86,82.2) and (267.8,74.14) .. (267.8,64.2) -- cycle ;
\draw   (358.8,63.2) .. controls (358.8,53.25) and (366.86,45.2) .. (376.8,45.2) .. controls (386.75,45.2) and (394.8,53.25) .. (394.8,63.2) .. controls (394.8,73.14) and (386.75,81.2) .. (376.8,81.2) .. controls (366.86,81.2) and (358.8,73.14) .. (358.8,63.2) -- cycle ;
\draw    (294,81.2) -- (320.06,130.43) ;
\draw [shift={(321,132.2)}, rotate = 242.1] [color={rgb, 255:red, 0; green, 0; blue, 0 }  ][line width=0.75]    (10.93,-3.29) .. controls (6.95,-1.4) and (3.31,-0.3) .. (0,0) .. controls (3.31,0.3) and (6.95,1.4) .. (10.93,3.29)   ;
\draw    (369,80.2) -- (344.87,130.39) ;
\draw [shift={(344,132.2)}, rotate = 295.68] [color={rgb, 255:red, 0; green, 0; blue, 0 }  ][line width=0.75]    (10.93,-3.29) .. controls (6.95,-1.4) and (3.31,-0.3) .. (0,0) .. controls (3.31,0.3) and (6.95,1.4) .. (10.93,3.29)   ;

\draw (326,142) node [anchor=north west][inner sep=0.75pt]    {$X$};
\draw (280,58) node [anchor=north west][inner sep=0.75pt]    {$Z$};
\draw (372,57) node [anchor=north west][inner sep=0.75pt]    {$Y$};

\end{tikzpicture}
\centering
\caption{Causal additive noise representation for $X$.}
\label{fig:causal_X}
        \end{subfigure}
        \hfill
        \begin{subfigure}[t]{0.49\textwidth}
        \vskip 0pt
		\begin{tikzpicture}[x=0.75pt,y=0.75pt,yscale=-1,xscale=1]

\draw   (314.8,148.2) .. controls (314.8,138.25) and (322.86,130.2) .. (332.8,130.2) .. controls (342.75,130.2) and (350.8,138.25) .. (350.8,148.2) .. controls (350.8,158.14) and (342.75,166.2) .. (332.8,166.2) .. controls (322.86,166.2) and (314.8,158.14) .. (314.8,148.2) -- cycle ;
\draw   (267.8,64.2) .. controls (267.8,54.25) and (275.86,46.2) .. (285.8,46.2) .. controls (295.75,46.2) and (303.8,54.25) .. (303.8,64.2) .. controls (303.8,74.14) and (295.75,82.2) .. (285.8,82.2) .. controls (275.86,82.2) and (267.8,74.14) .. (267.8,64.2) -- cycle ;
\draw   (358.8,63.2) .. controls (358.8,53.25) and (366.86,45.2) .. (376.8,45.2) .. controls (386.75,45.2) and (394.8,53.25) .. (394.8,63.2) .. controls (394.8,73.14) and (386.75,81.2) .. (376.8,81.2) .. controls (366.86,81.2) and (358.8,73.14) .. (358.8,63.2) -- cycle ;
\draw    (294,81.2) -- (320.06,130.43) ;
\draw [shift={(321,132.2)}, rotate = 242.1] [color={rgb, 255:red, 0; green, 0; blue, 0 }  ][line width=0.75]    (10.93,-3.29) .. controls (6.95,-1.4) and (3.31,-0.3) .. (0,0) .. controls (3.31,0.3) and (6.95,1.4) .. (10.93,3.29)   ;
\draw    (369,80.2) -- (344.87,130.39) ;
\draw [shift={(344,132.2)}, rotate = 295.68] [color={rgb, 255:red, 0; green, 0; blue, 0 }  ][line width=0.75]    (10.93,-3.29) .. controls (6.95,-1.4) and (3.31,-0.3) .. (0,0) .. controls (3.31,0.3) and (6.95,1.4) .. (10.93,3.29)   ;

\draw (326,142) node [anchor=north west][inner sep=0.75pt]    {$Z$};
\draw (280,58) node [anchor=north west][inner sep=0.75pt]    {$X$};
\draw (372,57) node [anchor=north west][inner sep=0.75pt]    {$Y$};

\end{tikzpicture}
		\centering
\caption{Causal additive noise representation for $Z$.}
\label{fig:causal_Z}
        \end{subfigure}
       \caption{\small \textit{Causal graphs} \cite[\S 3.2.2]{Morgan2015} for additive noise representations.} 
        \label{fig:causal_diagrams}
    \end{figure}
\begin{proposition}[Additive noise implies SIPD]
If either of the following conditions hold:
\begin{itemize}
\item $\left(Z, X\right)$ is generated via the causal representation \eqref{eq:additive_noise_X}, or
\item $\left(Z, X\right)$ is generated via the causal representation \eqref{eq:additive_noise_Z}, and moreover the copula of $\left(Z, X\right)$ is exchangeable (i.e. its distribution function $\mathcal{C}_{Z, X}\left(z, x\right)$ is permutation symmetric),
\end{itemize}
then $X$ is SIPD in $Z$.
\label{prop:additive_noise_implies_SIPD}
\end{proposition}
\begin{proof}
For the first causal additive representation, as $Z$ and $Y$ are independent, we have equality in law
\begin{equation}
\left[X\middle|Z=z\right]\underset{\mathrm{st}}{=}z+Y.
\end{equation}
Thus
\begin{align}
\operatorname{Pr}\left(X>x\middle|Z=z\right) &= \operatorname{Pr}\left(Z+Y>x\middle|Z=z\right) \\
&= \operatorname{Pr}\left(Y>x-z\middle|Z=z\right) \\
&\leq \operatorname{Pr}\left(Y>x-z'\middle|Z=z'\right) \\
&= \operatorname{Pr}\left(X>x\middle|Z=z'\right).
\end{align}
for all $z' \geq z$ in the support of $Z$. As for the second causal additive representation, assume without loss of generality that $\left(Z, X\right)$ is a copula, and denote $\mathcal{C}_{Z, X}\left(z, x\right) = \operatorname{Pr}\left(Z \leq z, X \leq x\right)$. By analogous arguments to the first causal additive representation, we have
\begin{equation}
\operatorname{Pr}\left(Z>z\middle|X=x\right) \leq \operatorname{Pr}\left(Z>z\middle|X=x'\right)
\label{eq:sipd_Z_in_X}
\end{equation}
for all $z, x, x' \in \left(0, 1\right)$ such that $x \leq x'$. Due to exchangeability, $\mathcal{C}_{Z, X}\left(z, x\right) = \mathcal{C}_{Z, X}\left(x, z\right)$ and we have the copula conditional CDF (computed via \cite[\S 2.1.3]{Joe2014}) as
\begin{align}
\operatorname{Pr}\left(Z>z\middle|X=x\right) &= \dfrac{\partial\mathcal{C}_{Z, X}\left(z,x\right)}{\partial x} \\
&= \dfrac{\partial\mathcal{C}_{Z, X}\left(x,z\right)}{\partial x} \\
&= \operatorname{Pr}\left(X>z\middle|Z=x\right).
\end{align}
Performing the same for the right-hand side of \eqref{eq:sipd_Z_in_X}, we get
\begin{equation}
\operatorname{Pr}\left(X>z\middle|Z=x\right) \leq \operatorname{Pr}\left(X>z\middle|Z=x'\right)
\end{equation}
for all $z, x, x' \in \left(0, 1\right)$ such that $x \leq x'$, which is the required condition for \eqref{eq:sipd}.
\end{proof}

Note that many classes of parametric copulas (with one or two parameters) are exchangeable \cite[\S 2.15]{Joe2014}, as are all members of the Archimedean family of copulas \cite{Harder2016}, so the SIPD condition is satisfied for a variety of models with additive noise.

%

\subsection{Additive Noise Representation for Gaussian Copulas}

Suppose observations $Z$ are generated by additive Gaussian noise to Gaussian $X$ in the way of \eqref{eq:additive_noise_Z}. Specifically, let
\begin{equation}
Z = X + Y,
\end{equation}
where $X \sim \mathcal{N}\left(0, 1\right)$ and $Y \sim \mathcal{N}\left(0, \xi^{2}\right)$ independent of $X$. The variable $\xi^{2}$ may be interpreted here as the \textit{noise-to-signal ratio}. Note that $\left(Z, X\right)$ is then bivariate Gaussian with correlation $\rho = \left(1 + \xi^{2}\right)^{-1/2}$ and naturally, a bivariate Gaussian has a bivariate Gaussian copula. Hence this representation corresponds with the Gaussian copula success probability $p_{\mathrm{success}}^{\mathcal{NC}\left(1/\sqrt{1 + \xi^{2}}\right)}\left(n, m, \alpha\right)$, in terms of the noise-to-signal ratio.


\section{Randomised Selection Size}
\label{sec:randomised_selection_size}

Consider the alternative selection rule: to select all the candidates below the threshold $z_{m/n}^{*}$, where $z_{m/n}^{*}$ is the $100m/n$ percentile of $Z$. This selection rule gives a randomised selection size that is binomial distributed with parameters $n$ and $m/n$, and hence has an expectation of $m$, and asymptotic variance of $\lim_{n\to\infty}n\left(m/n\right)\left(1-m/n\right) = m$. The success probability with this alternative selection rule, defined in the analogous way to \eqref{eq:ord-opt-defn}, is equal to $1 - \left(1 - C_{Z,X}\left(m/n,\alpha\right)\right)^{n}$, which is explained as follows. Note that there will be a success if at least one of the candidates satisfies both $Z_{i} \leq z_{m/n}^{*}$ and $X_{i} \leq x_{\alpha}^{*}$; the former is the condition that it is selected the first place, and the latter is the condition it actually performs acceptably well. Thus, the failure probability is the probability that none of the samples satisfies both $Z_{i} \leq z_{m/n}^{*}$ and $X_{i} \leq x_{\alpha}^{*}$. Since
\begin{equation}
c := C_{Z,X}\left(m/n,\alpha\right) = \operatorname{Pr}\left(X\leq x_{\alpha}^{*}, Z \leq z_{m/n}^{*}\right),
\end{equation}
then the failure probability is $\left(1 - c\right)^{n}$, and so the success probability is $1 - \left(1 - c\right)^{n}$. We compute the limit of this success probability as $n \to \infty$ for fixed $m$. Assume without loss of generality that $\left(Z, X\right)$ is a copula. Then we have
\begin{align}
\lim_{n\to\infty}\left[1-\left(1-c\right)^{n}\right] &= \lim_{n\to\infty}\left[1-\left(1-\operatorname{Pr}\left(X\leq x_{\alpha}^{*}\middle|Z\leq z_{m/n}^{*}\right)\operatorname{Pr}\left(Z\leq z_{m/n}^{*}\right)\right)^{n}\right] \\
&= 1-\lim_{n\to\infty}\left[\left(1-\dfrac{\operatorname{Pr}\left(X\leq x_{\alpha}^{*}\middle|Z\leq z_{m/n}^{*}\right)m}{n}\right)^{n}\right] \\
&= 1-\exp\left(-m\lim_{n\to\infty}\operatorname{Pr}\left(X\leq x_{\alpha}^{*}\middle|Z\leq z_{m/n}^{*}\right)\right) \\
&= 1-\exp\left(-mC_{X|Z}\left(\alpha\middle|0\right)\right),
\end{align}
which depends on the copula conditional boundary CDF $C_{X|Z}\left(\alpha\middle|0\right) \leq 1$, thus
\begin{equation}
1-\exp\left(-mC_{X|Z}\left(\alpha\middle|0\right)\right)\leq1-e^{-m}
\end{equation}
which is always strictly less than one for finite $m$. On the other hand, the fixed size selection rule has the limiting success probability from Proposition \ref{thm:p_success_limiting}\ref{thm:p_success_limiting_n_to_infty} of
\begin{equation}
\lim_{n\to\infty}p_{\mathrm{success}}^{\mathcal{C}_{Z, X}}\left(n,m,\alpha\right)=1-\left(1-C_{X|Z}\left(\alpha\middle|0\right)\right)^{m},
\end{equation}
which can achieve one for finite $m$, provided $C_{X|Z}\left(\alpha\middle|0\right) = 1$. Furthermore, this is shown to be asymptotically strictly greater than that for the randomised selection size. Let $c_{\alpha} := C_{X|Z}\left(\alpha\middle|0\right)$, and use the fact that $\log\left(1 - c_{\alpha}\right) \leq -c_{\alpha}$ for all $c_{\alpha} \geq 0$ (refer to Lemma \ref{lem:log-concave-bounds}) so that
\begin{equation}
1 - e^{-mc_{\alpha}} \leq 1 - \left(1 - c_{\alpha}\right)^{m},
\end{equation} 
with equality if and only if $c_{\alpha} = 0$, i.e. both sides are equal to zero. However the lower bound in Proposition \ref{thm:general_bounds_OO} implies a strictly positive success probability, so it must be that
\begin{equation}
1 - e^{-mc_{\alpha}} < 1 - \left(1 - c_{\alpha}\right)^{m}.
\end{equation}

\section{Proofs of Main Results}
\label{sec:proofs}

\subsection{Proof of Theorem \ref{thm:p_success_m-dimensional_integral}}
Observe that
\begin{equation}
\min\left\{ X_{\left\langle 1\right\rangle },\dots,X_{\left\langle m\right\rangle }\right\} \leq x_{\alpha}^{*}
\end{equation}
if and only if
\begin{equation}
\min\left\{ F_{X}\left(X_{\left\langle 1\right\rangle }\right),\dots,F_{X}\left(X_{\left\langle m\right\rangle }\right)\right\} \leq\alpha.
\end{equation}
Moreover, the ordering of $F_{Z}\left(Z_{1}\right), \dots, F_{Z}\left(Z_{n}\right)$ compared to $Z_{1}, \dots, Z_{n}$ is unchanged. Therefore using the probability integral transform, we can assume without loss of generality that $\left(Z, X\right)$ is a copula distribution, so that $Z$ and $X$ are each $\operatorname{Uniform}\left(0, 1\right)$ distributed. With the law of total probability, write
\begin{multline}
p_{\mathrm{success}}^{F_{Z,X}}\left(n,m,\alpha\right) = \int_{0}^{1}\dots\int_{0}^{1}\operatorname{Pr}\left(\min\left\{ X_{\left\langle 1\right\rangle },\dots,X_{\left\langle m\right\rangle }\right\} \leq\alpha\middle|Z_{1:n}=z_{1},\dots,Z_{m:n}=z_{m}\right) \\
\times f_{Z_{1:n},\dots,Z_{m:n}}\left(z_{1},\dots,z_{m}\right)\mathrm{d}z_{1}\dots \mathrm{d}z_{m},
\label{eq:thm_psuccess_m-dimensional_integral_copula}
\end{multline}
where
\begin{equation}
f_{Z_{1:n},\dots,Z_{m:n}}\left(z_{1},\dots,z_{m}\right)=\dfrac{n!}{\left(n-m\right)!}\left(1-z_{m}\right)^{n-m}\mathbb{I}_{\left\{ z_{1}\leq\dots\leq z_{m}\right\} }
\label{eq:thm_psuccess_m-dimensional_integral_order_statistics_pdf}
\end{equation}
is the joint PDF of the first $m$ order statistics of the i.i.d. $\operatorname{Uniform}\left(0, 1\right)$ sample $\left(Z_{1}, \dots, Z_{n}\right)$, for which the form may be deduced via \cite[\S 2.2]{David2005}. Further note that since each pair $\left(Z_{i}, X_{i}\right)$ is sampled independently, then each $X_{\left\langle j \right\rangle}$ is conditionally independent of all the variables
\begin{equation}
\left(X_{\left\langle 1\right\rangle },\dots,X_{\left\langle j-1\right\rangle },X_{\left\langle j+1\right\rangle },\dots,X_{\left\langle m\right\rangle },Z_{1:n},\dots,Z_{\left(j-1\right):n},Z_{\left(j+1\right):n},\dots,Z_{m:n}\right),
\end{equation}
given $Z_{j:m}$. Denote the event $\mathcal{Z} = \left\{Z_{1:n}=z_{1},\dots,Z_{m:n}=z_{m}\right\}$ for brevity. Then
\begin{align}
\operatorname{Pr}\left(\min\left\{ X_{\left\langle 1\right\rangle },\dots,X_{\left\langle m\right\rangle }\right\} \leq\alpha\middle|\mathcal{Z}\right) &= 1-\operatorname{Pr}\left(X_{\left\langle 1\right\rangle }>\alpha,\dots,X_{\left\langle m\right\rangle }>\alpha\middle|\mathcal{Z}\right) \\
&= 1-\prod_{j=1}^{m}\operatorname{Pr}\left(X_{\left\langle j\right\rangle }>\alpha\middle|Z_{j:n}=z_{j}\right) \\
&= 1-\prod_{j=1}^{m}\operatorname{Pr}\left(X>\alpha\middle|Z=z_{j}\right) \\
&= 1-\prod_{j=1}^{m}\left(1-\operatorname{Pr}\left(X\leq\alpha\middle|Z=z_{j}\right)\right) \\
&= 1-\prod_{j=1}^{m}\left(1-\mathcal{C}_{X|Z}\left(\alpha\middle|z_{j}\right)\right).\label{eq:thm_psuccess_m-dimensional_integral_conditional_prob}
\end{align}
Plugging \eqref{eq:thm_psuccess_m-dimensional_integral_order_statistics_pdf} and \eqref{eq:thm_psuccess_m-dimensional_integral_conditional_prob} into \eqref{eq:thm_psuccess_m-dimensional_integral_copula}, we have the result claimed. \qed

\subsection{Proof of Proposition \ref{thm:optimality_horse_race}}

Starting from \eqref{eq:thm_psuccess_m-dimensional_integral_copula} and applying \eqref{eq:thm_psuccess_m-dimensional_integral_conditional_prob}, we have
\begin{align}
p_{\mathrm{success}}^{\mathcal{C}_{Z, X}}\left(n,m,\alpha\right) &= \begin{multlined}[t] \int_{0}^{1}\dots\int_{0}^{1}\operatorname{Pr}\left(\min\left\{ X_{\left\langle 1\right\rangle },\dots,X_{\left\langle m\right\rangle }\right\} \leq\alpha\middle|Z_{1:n}=z_{1},\dots,Z_{m:n}=z_{m}\right) \\
\times f_{Z_{1:n},\dots,Z_{m:n}}\left(z_{1},\dots,z_{m}\right)\mathrm{d}z_{1}\dots \mathrm{d}z_{m}
\end{multlined} \\
&= 1-\mathbb{E}_{Z_{1:n},\dots,Z_{m:n}}\left[\prod_{j=1}^{m}\operatorname{Pr}\left(X>\alpha\middle|Z=Z_{j:n}\right)\right].
\end{align}
Consider some arbitrary ordered selection of size $m$ from $\left(Z_{1}, \dots, Z_{n}\right)$, denoted $\left(Z_{1}', \dots, Z_{m}'\right)$. Then we have multivariate stochastic dominance \cite[\S 6.B.1]{Shaked2007}
\begin{equation}
\left(Z_{1:n},\dots,Z_{m:n}\right)\underset{\mathrm{st}}{\preceq}\left(Z_{1}',\dots,Z_{m}'\right),
\end{equation}
since
\begin{equation}
Z_{1:n} \leq Z_{1}', \quad Z_{2:n} \leq Z_{2}', \quad \dots,\quad Z_{m:n} \leq Z_{m}'
\end{equation}
for every realisation of $\left(Z_{1}, \dots, Z_{n}\right)$, so \cite[Theorem 6.B.1]{Shaked2007} holds. Applying the SIPD condition, $\prod_{j=1}^{m}\operatorname{Pr}\left(X>\alpha\middle|Z=z_{j}\right)$ is a non-decreasing function in $\left(z_{1}, \dots, z_{m}\right)$, so by the definition of stochastic dominance \cite[\S 6.B.1]{Shaked2007}, we have
\begin{equation}
\mathbb{E}_{Z_{1:n},\dots,Z_{m:n}}\left[\prod_{j=1}^{m}\operatorname{Pr}\left(X>\alpha\middle|Z=Z_{j:n}\right)\right]\leq\mathbb{E}_{Z_{1}',\dots,Z_{m}'}\left[\prod_{j=1}^{m}\operatorname{Pr}\left(X>\alpha\middle|Z=Z_{j}'\right)\right].
\end{equation}
Hence the success probability is maximised when $\left(Z_{1}', \dots, Z_{m}'\right)$ is chosen as the first $m$ order statistics. \qed

\subsection{Proof of Proposition \ref{thm:monotonicity_OO}}

For the proof of Proposition \ref{thm:monotonicity_OO}, we require the following lemma.
\begin{lemma}
Let $\mathbf{Z}_{\left[m\right]:n} := \left(Z_{1:n}, \dots, Z_{m:n}\right)$ denote the joint first $m$ order statistics of an i.i.d. sample of size $n$ from parent distribution $Z$. Then
\begin{equation}
\mathbf{Z}_{\left[m\right]:\left(n + 1\right)} \underset{\mathrm{st}}{\preceq} \mathbf{Z}_{\left[m\right]:n}.
\label{eq:stoch-dom-nplus1}
\end{equation}
\label{lem:ord-stat-stoch-dom-increasing-n}
\end{lemma}
\begin{proof}
Consider the following construction on the same probability space. Form an i.i.d. sample of size $n + 1$ from $Z$ and take the first $m$ order statistics. This will be equal in law to $\mathbf{Z}_{\left[m\right]:\left(n + 1\right)}$. Now delete one element uniformly at random, and re-compute the first $m$ order statistics. This will be equal in law to $\mathbf{Z}_{\left[m\right]:n}$. Moreover, for every realisation (denoted in lowercase), we have
\begin{equation}
\left(z_{1:\left(n + 1\right)}, \dots, z_{m:\left(n + 1\right)}\right) \leq \left(z_{1:n}, \dots, z_{m:n}\right).
\end{equation}
Therefore from the characterisation of stochastic dominance in \cite[Theorem 6.B.1]{Shaked2007}, \eqref{eq:stoch-dom-nplus1} holds. 
\end{proof}

For \ref{thm:monotonicity_OO_n}, it follows from Lemma \ref{lem:ord-stat-stoch-dom-increasing-n}, via the same technique and analogous arguments as in the proof of Proposition \ref{thm:optimality_horse_race}. \\

For \ref{thm:monotonicity_OO_m}, since
\begin{equation}
\min\left\{ X_{\left\langle 1\right\rangle },\dots,X_{\left\langle m\right\rangle }\right\} \leq\min\left\{ X_{\left\langle 1\right\rangle },\dots,X_{\left\langle m+1\right\rangle }\right\},
\end{equation}
then
\begin{align}
p_{\mathrm{success}}^{\mathcal{C}_{Z, X}}\left(n,m,\alpha\right) &= \operatorname{Pr}\left(\min\left\{ X_{\left\langle 1\right\rangle },\dots,X_{\left\langle m\right\rangle }\right\} \leq\alpha\right) \\
&\leq \operatorname{Pr}\left(\min\left\{ X_{\left\langle 1\right\rangle },\dots,X_{\left\langle m+1\right\rangle }\right\} \leq\alpha\right) \\
&= p_{\mathrm{success}}^{\mathcal{C}_{Z, X}}\left(n,m+1,\alpha\right).
\end{align}
For \ref{thm:monotonicity_OO_alpha}, by De Morgan's laws (i.e. complement of the union is the intersection of the complements), put the definition of $p_{\mathrm{success}}^{\mathcal{C}_{Z, X}}\left(n, m, \alpha\right)$ in terms of
\begin{equation}
p_{\mathrm{success}}^{\mathcal{C}_{Z, X}}\left(n, m, \alpha\right) = 1 - \operatorname{Pr}\left(\bigcap_{i = 1}^{m}\left\{X_{\left\langle i\right\rangle} > x_{\alpha}^{*}\right\}\right).
\end{equation}
Then apply the properties that $x_{\alpha}^{*}$ is non-decreasing in $\alpha$ and $\operatorname{Pr}\left(\bigcap_{i = 1}^{m}\left\{X_{\left\langle i\right\rangle} > x_{\alpha}^{*}\right\}\right)$ is non-increasing in $x_{\alpha}^{*}$. \qed

\subsection{Proof of Proposition \ref{thm:general_bounds_OO}}

For the lower bound, consider the success probability where the selection of size $m$ is uniformly random without replacement from $\left(Z_{1}, \dots, Z_{n}\right)$. This selection is identical in law to $\left(Z_{1}, \dots, Z_{m}\right)$, which are i.i.d. Then the success probability can be computed by
\begin{align}
\operatorname{Pr}\left(\min\left\{ X_{1},\dots,X_{m}\right\} \leq x_{\alpha^{*}}\right) &= 1-\prod_{j=1}^{m}\operatorname{Pr}\left(X_{j}>x_{\alpha}^{*}\right) \\
&= 1-\left(1-\alpha\right)^{m}.
\end{align}
Based on similar arguments provided from the proof of Proposition \ref{thm:optimality_horse_race}, it follows that this lower bounds $p_{\mathrm{success}}^{\mathcal{C}_{Z, X}}$. \\

For the upper bound, we can use the fact
\begin{equation}
\min_{i\in\left\{ 1,\dots,n\right\} }X_{i}\leq\min\left\{ X_{\left\langle 1\right\rangle },\dots,X_{\left\langle m\right\rangle }\right\}
\end{equation}
to bound
\begin{align}
p_{\mathrm{success}}^{\mathcal{C}_{Z, X}}\left(n,m,\alpha\right) &= \operatorname{Pr}\left(\min\left\{ X_{\left\langle 1\right\rangle },\dots,X_{\left\langle m\right\rangle }\right\}\leq x_{\alpha}^{*}\right) \\
&\leq \operatorname{Pr}\left(\min_{i\in\left\{ 1,\dots,n\right\} }X_{i}\leq x_{\alpha}^{*}\right) \\
&= 1-\prod_{i=1}^{n}\operatorname{Pr}\left(X_{i}>\alpha\right) \\
&= 1-\left(1-\alpha\right)^{n}.
\end{align}
\qed

\subsection{Proof of Proposition \ref{thm:p_success_limiting}}

For cases \ref{thm:p_success_limiting_m_equals_n}, \ref{thm:p_success_limiting_alpha_1}, \ref{thm:p_success_limiting_alpha_to_0}, the results are immediate from applying the general upper and lower bounds in Proposition \ref{thm:general_bounds_OO}. For \ref{thm:p_success_limiting_comonotonic}, assuming without loss of generality that $\left(Z, X\right)$ is a copula and applying the property of comotonicity, then
\begin{align}
p_{\mathrm{success}}^{\mathcal{C}_{Z, X}}\left(n,m,\alpha\right) &= \operatorname{Pr}\left(\min\left\{ X_{\left\langle 1\right\rangle },\dots,X_{\left\langle m\right\rangle }\right\} \leq\alpha\right) \\
&= \operatorname{Pr}\left(\min\left\{ Z_{1:n},\dots,Z_{m:n}\right\} \leq\alpha\right) \\
&= \operatorname{Pr}\left(Z_{1:n}\leq\alpha\right) \\
&= 1-\prod_{i=1}^{n}\operatorname{Pr}\left(Z_{i}>\alpha\right) \\
&= 1-\left(1-\alpha\right)^{n}.
\end{align}
For \ref{thm:p_success_limiting_independent}, by independence of $Z$ and $X$ we have from \eqref{eq:thm_psuccess_m-dimensional_integral_conditional_prob}:
\begin{align}
1-\prod_{j=1}^{m}\left(1-\operatorname{Pr}\left(X\leq\alpha\middle|Z=z_{j}\right)\right) &= 1-\prod_{j=1}^{m}\left(1-\operatorname{Pr}\left(X\leq\alpha\right)\right) \\
&= 1-\left(1-\alpha\right)^{m}.
\end{align}
This can be taken outside of the integral in \eqref{eq:thm_psuccess_m-dimensional_integral_copula} (where the integral evaluates to one), so the result follows. Case \ref{thm:p_success_limiting_m_to_infty} follows from the lower bound in Proposition \ref{thm:general_bounds_OO}. Lastly for \ref{thm:p_success_limiting_n_to_infty}, as $n \to \infty$, the joint density \eqref{eq:thm_psuccess_m-dimensional_integral_order_statistics_pdf} converges to the Dirac delta function in argument $z_{m}$. Considering the region of integration in the expression \eqref{eq:p_success_m-dimensional_integral} where $0 \leq z_{1} \leq \dots \leq z_{m}$, we have that $z_{m} = 0$ implies $z_{1} = \dots = z_{m - 1} = 0$. Therefore
\begin{align}
\lim_{n\to\infty}p_{\mathrm{success}}\left(n,m,\alpha\right) &= 1-\prod_{j=1}^{m}\left(1-\operatorname{Pr}\left(X\leq\alpha\middle|Z=0\right)\right) \\
&= 1 - \left(1 - \mathcal{C}_{X|Z}\left(\alpha\middle|0\right)\right)^{m}.
\end{align} \qed

\subsection{Proof of Theorem \ref{thm:constructed-lower-bound}}

For the proof of Theorem \ref{thm:constructed-lower-bound}, we require the following lemmas.

\begin{lemma}
We have
\begin{equation}
-p\log 4 \leq \log\left(1 - p\right) \leq -p,
\end{equation}
where the lower bound applies for all $p \in \left[0, \frac{1}{2}\right]$, and the upper bound applies for all $p \geq 0$.
\label{lem:log-concave-bounds}
\end{lemma}
\begin{proof}
The lower bound can be established over $p \in \left[0, \frac{1}{2}\right]$ via concavity of $\log\left(1 - p\right)$, i.e. line secants lie below the graph. The upper bound can also be established over $p \geq 0$ via concavity.
\end{proof}

\begin{lemma}
For the Gaussian $Q$-function given by $Q\left(x\right) = 1 - \Phi\left(x\right)$, we have
\begin{equation}
\mathfrak{c}_{1}\exp\left(-\mathfrak{c}_{2}x^{2}\right) \leq Q\left(x\right) \leq \frac{1}{2}\exp\left(-\dfrac{x^{2}}{2}\right)
\end{equation}
over $x \geq 0$, with
\begin{gather}
\mathfrak{c}_{1} = \dfrac{1}{2} - \dfrac{\omega}{\pi} \\
\mathfrak{c}_{2} = \dfrac{\cot{\omega}}{\pi - 2\omega}
\end{gather}
for any $\omega \in \left(0, \frac{\pi}{2}\right)$.
\label{lem:q-function-bounds}
\end{lemma}
\begin{proof}
The lower bound is due to \cite[Equation (2)]{Wu2018} and the upper bound is found in \cite[Equation (5)]{Chiani2003}.
\end{proof}

\begin{lemma}
Let $Z_{1:n}$ denote the first order statistic of an i.i.d. standard Gaussian sample of size $n$. For any $\omega \in \left(0, \frac{\pi}{2}\right)$, let
\begin{gather}
\mathfrak{c}_{1} = \dfrac{1}{2} - \dfrac{\omega}{\pi} \\
\mathfrak{c}_{2} = \dfrac{\cot{\omega}}{\pi - 2\omega}.
\end{gather}
Consider $\widehat{Z}_{1:n} \sim \mathcal{N}\left(\mu_{n}, \sigma_{n}^{2}\right)$ where
\begin{gather}
\mu_{n} = -\sqrt{\dfrac{\log\left(n\mathfrak{c}_{1}\right)}{\mathfrak{c}_{2}}} \label{eq:mu-constructed2} \\
\sigma_{n}^{2} = \dfrac{-\log\log 2}{2\mathfrak{c}_{2}\left(\log\left(n\mathfrak{c}_{1}\right) - \log\log 2\right)}. \label{eq:sigma-constructed2}
\end{gather}
Then there exists some integer $\widetilde{n}\left(\omega\right)$ such that for all $n \geq \widetilde{n}\left(\omega\right)$, we have $Z_{1:n} \underset{\mathrm{st}}{\preceq} \widehat{Z}_{1:n}$, i.e. $\widehat{Z}_{1:n}$ stochastically dominates $Z_{1:n}$.
\label{lem:stoch-dom-constructed}
\end{lemma}
\begin{proof}

If $\widehat{Z}_{1:n}$ stochastically dominates $Z_{1:n}$, then by definition $\operatorname{Pr}\left(\widehat{Z}_{1:n} \geq z\right) \geq \operatorname{Pr}\left(Z_{1:n} \geq z\right)$ for all $z \in \mathbb{R}$. Or in terms of the Gaussian $Q$-function $Q\left(z\right) = 1 - \Phi\left(z\right)$, we require
\begin{equation}
Q\left(z\right)^{n} \leq Q\left(\dfrac{z - \mu_{n}}{\sigma_{n}}\right)
\label{eq:q-function-stoch-dom}
\end{equation}
for all $z \in \mathbb{R}$, where the left-hand side follows from well-known form of the distribution for the first order statistic \cite[Equation (2.2.11)]{Arnold2008}. The idea is to show that this bound holds over three different intervals whose union is $\mathbb{R}$, being $\left(-\infty, \mu_{n}\right]$, $\left[\mu_{n}, 0\right]$ and $\left[0, \infty\right)$. We begin with $z \in \left(-\infty, \mu_{n}\right]$. Since $\mu_{n} \leq 0$, then via the lower bound in Lemma \ref{lem:q-function-bounds}
\begin{equation}
Q\left(z\right)^{n} \leq \left(1 - \mathfrak{c}_{1}\exp\left(-\mathfrak{c}_{2}z^{2}\right)\right)^{n}.
\end{equation}
Since $0 \leq \mathfrak{c}_{1}\exp\left(-\mathfrak{c}_{2}z^{2}\right) \leq 1/2$, then putting $p = \mathfrak{c}_{1}\exp\left(-\mathfrak{c}_{2}z^{2}\right)$ in the upper bound from Lemma \ref{lem:log-concave-bounds}, we get
\begin{align}
Q\left(z\right)^{n} &\leq \exp\left(-n\mathfrak{c}_{1}\exp\left(-\mathfrak{c}_{2}z^{2}\right)\right) \\
&= \exp\left(-\exp\left(-\left(\mathfrak{c}_{2}z^{2}-\log\left(n\mathfrak{c}_{1}\right)\right)\right)\right).
\end{align}
Now using the upper bound in Lemma \ref{lem:q-function-bounds}, we have for $z \leq \mu_{n}$:
\begin{equation}
1-\dfrac{1}{2}\exp\left(-\dfrac{\left(z-\mu_{n}\right)^{2}}{2\sigma_{n}^{2}}\right)\leq Q\left(\dfrac{z-\mu_{n}}{\sigma_{n}}\right).
\end{equation}
The lower bound in Lemma \ref{lem:log-concave-bounds} implies $\exp\left(-p\log 4\right) \leq 1 - p$. Applying this with $p = \frac{1}{2}\exp\left(-\frac{\left(z-\mu_{n}\right)^{2}}{2\sigma_{n}^{2}}\right)$ and after some manipulation, we arrive at
\begin{equation}
Q\left(\dfrac{z-\mu_{n}}{\sigma_{n}}\right) \geq \exp\left(-\exp\left(-\left(\dfrac{\left(z-\mu_{n}\right)^{2}}{2\sigma_{n}^{2}}-\log\log2\right)\right)\right).
\end{equation}
Thus a sufficient condition for $Q\left(z\right)^{n} \leq Q\left(\frac{z - \mu_{n}}{\sigma_{n}}\right)$ over $z \in \left(-\infty, \mu_{n}\right]$ is
\begin{equation}
\exp\left(-\exp\left(-\left(\mathfrak{c}_{2}z^{2}-\log\left(n\mathfrak{c}_{1}\right)\right)\right)\right) \leq \exp\left(-\exp\left(-\left(\dfrac{\left(z-\mu_{n}\right)^{2}}{2\sigma_{n}^{2}}-\log\log2\right)\right)\right)
\end{equation}
or equivalently,
\begin{equation}
\left(\dfrac{1}{\sigma_{n}^{2}}-2\mathfrak{c}_{2}\right)z^{2}-2\dfrac{\mu_{n}}{\sigma_{n}^{2}}z+\dfrac{\mu_{n}^{2}}{\sigma_{n}^{2}}+2\log\left(n\mathfrak{c}_{1}\right)-2\log\log2\geq0.
\end{equation}
The roots of this quadratic are at
\begin{equation}
z = \dfrac{\mu_{n}\pm\sqrt{\mu_{n}^{2}-\left(1-2\mathfrak{c}_{2}\sigma_{n}^{2}\right)\left(\mu_{n}^{2}+2\sigma_{n}^{2}\log\left(n\mathfrak{c}_{1}\right)-2\sigma_{n}^{2}\log\log2\right)}}{1-2\mathfrak{c}_{2}\sigma_{n}^{2}}
\end{equation}
with discriminant $\Delta$ calculated by
\begin{equation}
\Delta = \sigma_{n}^{4}\left(4\mathfrak{c}_{2}\log\left(n\mathfrak{c}_{1}\right)-4\mathfrak{c}_{2}\log\log2\right)+\sigma_{n}^{2}\left(2\mathfrak{c}_{2}\mu_{n}^{2}-2\log\left(n\mathfrak{c}_{1}\right)+2\log\log2\right).
\end{equation}
Under the same choice of $\omega$, note $2\mathfrak{c}_{2}\mu_{n}^{2}=2\log\left(n\mathfrak{c}_{1}\right)$ and the discriminant becomes
\begin{equation}
\Delta=\sigma_{n}^{4}\left(4\mathfrak{c}_{2}\log\left(n\mathfrak{c}_{1}\right)-4\mathfrak{c}_{2}\log\log2\right)+\sigma_{n}^{2}\left(2\log\log2\right).
\end{equation}
The quadratic inequality is satisfied everywhere if the discriminant is non-positive, so put $\Delta = 0$ and taking the positive solution for $\sigma_{n}^{2}$, giving
\begin{equation}
\sigma_{n}^{2} = \dfrac{-\log\log 2}{2\mathfrak{c}_{2}\left(\log\left(n\mathfrak{c}_{1}\right) - \log\log 2\right)}.
\end{equation}
Therefore the inequality is satisfied provided $n\mathfrak{c}_{1} > 1$, which occurs for sufficiently large $n$, since $\mathfrak{c}_{1} > 0$. Next we show that the stochastic dominance condition is satisfied for $z \in \left[\mu_{n}, 0\right]$, under the proposed choice of $\mu_{n}$ and $\sigma_{n}$ above. Over this interval, we can use the same upper bound on $Q\left(z\right)^{n}$ as before, and now we have the lower bound
\begin{equation}
Q\left(\dfrac{z-\mu_{n}}{\sigma_{n}}\right) \geq \exp\left(-\left(\mathfrak{c}_{2}\left(\dfrac{z-\mu_{n}}{\sigma_{n}}\right)^{2}-\log \mathfrak{c}_{1}\right)\right).
\end{equation}
Thus we want to show that
\begin{equation}
n\mathfrak{c}_{1}\exp\left(-\mathfrak{c}_{2}z^{2}\right)\geq \mathfrak{c}_{2}\left(\dfrac{z-\mu_{n}}{\sigma_{n}}\right)^{2}-\log \mathfrak{c}_{1}.
\end{equation}
Fix $z$, and recognise that $\frac{\mu_{n}^{2}}{\sigma_{n}^{2}} = O\left(\left(\log n\right)^{2}\right)$ in the right-hand side, while the left-hand side is $O\left(n\right)$. Therefore
\begin{equation}
O\left(n\right) \geq \left(\left(\log n\right)^{2}\right)
\end{equation}
since $O\left(e^{n}\right) \geq O\left(n^{2}\right)$. Lastly for the interval $z \in \left[0, \infty\right)$, we use the upper bound in Lemma \ref{lem:q-function-bounds} to give
\begin{equation}
Q\left(z\right)^{n} \leq \dfrac{1}{2^{n}}\exp\left(-\dfrac{nz^{2}}{2}\right)
\end{equation}
and we can use the same lower bound as in the preceding interval. In the same vein as above, we want to show
\begin{equation}
\left(\dfrac{n}{2}-\dfrac{\mathfrak{c}_{2}}{\sigma_{n}^{2}}\right)z^{2}-\dfrac{2\mathfrak{c}_{2}\mu_{n}}{\sigma_{n}^{2}}+n\log2+\dfrac{\mu_{n}^{2}}{\sigma_{n}^{2}}-\log \mathfrak{c}_{1}\geq0.
\end{equation}
The discriminant of the quadratic is non-positive when
\begin{equation}
\left(\dfrac{n}{2}-\dfrac{\mathfrak{c}_{2}}{\sigma_{n}^{2}}\right)\left(n\log2+\dfrac{\mu_{n}^{2}}{\sigma_{n}^{2}}-\log \mathfrak{c}_{1}\right)\geq\dfrac{\mathfrak{c}_{2}^{2}\mu_{n}^{2}}{\sigma_{n}^{4}}.
\label{eq:right-interval-stoch-dom}
\end{equation}
The left-hand side is $O\left(n^{2}\right)$ and the right-hand side is $O\left(\left(\log n\right)^{3}\right)$, thus this inequality is also satisfied for sufficiently large $n$.

\end{proof}


To establish the right inequality in \eqref{eq:success-prob-copula-lb}, assume without loss in generality that $\left(Z, X\right)$ are bivariate standard Gaussian with correlation $\rho$, as in \eqref{eq:bivariate-gaussian-copula}. By establishing $Z_{1:n} \underset{\mathrm{st}}{\preceq} \widehat{Z}_{1:n}$ via Lemma \ref{lem:stoch-dom-constructed}, then via the same technique as in the proof of Proposition \ref{thm:optimality_horse_race}, we have
\begin{align}
p_{\mathrm{success}}^{\mathcal{NC}\left(p\right)}\left(n,1,\alpha\right) &= \int_{-\infty}^{\infty}\operatorname{Pr}\left(X\leq\Phi^{-1}\left(\alpha\right)\middle|Z=z\right)f_{Z_{1:n}}\left(z\right)\mathrm{d}z \\
&\geq \int_{-\infty}^{\infty}\operatorname{Pr}\left(X\leq\Phi^{-1}\left(\alpha\right)\middle|Z=z\right)f_{\widehat{Z}_{1:n}}\left(z\right)\mathrm{d}z, \label{eq:marginalisation_integral}
\end{align}
where $f_{Z_{1:n}}\left(\cdot\right)$ is the density of the first order statistic of the standard Gaussian, and $f_{\widehat{Z}_{1:n}}\left(\cdot\right)$ is the density of $\mathcal{N}\left(\mu_{n}, \sigma^{2}_{n}\right)$. Using well-known conditioning formulae for Gaussians \cite[Equation (A.6)]{Rasmussen2006}, we also have
\begin{equation}
\left[X\middle|Z=z\right]\sim\mathcal{N}\left(\rho z,1-\rho^{2}\right).
\end{equation}
Thus, the integral \eqref{eq:marginalisation_integral} may be computed analytically with well-known marginalisation formulae for linear-Gaussian systems \cite[Theorem 2.3.1]{Krishnamurthy2016}. In particular, if $\widehat{Z}_{1:n} \sim \mathcal{N}\left(\mu_{n}, \sigma^{2}_{n}\right)$ and $\left[X\middle|\widehat{Z}_{1:n}=z\right]\sim\mathcal{N}\left(\rho z,1-\rho^{2}\right)$, then $X \sim \mathcal{N}\left(\rho \mu_{n}, 1 - \rho^{2} + \rho^{2}\sigma_{n}^{2}\right)$. Hence
\begin{equation}
\int_{-\infty}^{\infty}\operatorname{Pr}\left(X\leq\Phi^{-1}\left(\alpha\right)\middle|Z=z\right)f_{\widehat{Z}_{1:n}}\left(z\right)\mathrm{d}z = \Phi\left(\dfrac{\Phi^{-1}\left(\alpha\right)-\rho\mu_{n}}{\sqrt{1-\rho^{2}+\rho^{2}\sigma_{n}^{2}}}\right).
\end{equation}
The left inequality in \eqref{eq:success-prob-copula-lb} is a result of monotonicity in $m$ from Proposition \ref{thm:monotonicity_OO}\ref{thm:monotonicity_OO_m}. \qed

\section{Algorithms}
\label{sec:algorithms}

The optimised lower bound in \eqref{eq:success-prob-copula-lb-optimised} can be implemented numerically. This is done by using sufficient conditions found in the proof of Lemma \ref{lem:stoch-dom-constructed} to check whether $n \geq \widetilde{n}\left(\omega\right)$ for a given $n$ and $\omega$. We are required to check whether the inequality \eqref{eq:q-function-stoch-dom} is satisfied over each of the intervals $\left(\infty, \mu_{n}\right]$, $\left[\mu_{n}, 0\right]$ and $\left[0, \infty\right)$. The inequality is satisfied over $\left(\infty, \mu_{n}\right]$ by construction provided $n\mathfrak{c}_{1} > 1$, whereas \eqref{eq:right-interval-stoch-dom} contains the sufficient condition for the interval $\left[0, \infty\right)$. As for the bounded interval $\left[\mu_{n}, 0\right]$, we can directly evaluate (up to the available numerical precision) whether \eqref{eq:q-function-stoch-dom} is satisfied. Pseudocode to implement this numerical certificate is provided in Algorithm \ref{alg:test}. Using this certificate, we can implement the optimised lower bound \eqref{eq:success-prob-copula-lb-optimised}, with pseudocode for this found in Algorithm \ref{alg:lb}. \\

\begin{algorithm}[!htb]
\caption{Numerical certification of sufficient conditions for $n \geq \widetilde{n}\left(\omega\right)$ in Theorem \ref{thm:constructed-lower-bound}}
\label{alg:test}
\begin{algorithmic}[1]
\Function{NumericalCert}{$n$, $\omega$}
\State $\mathfrak{c}_{1} \gets \dfrac{1}{2} - \dfrac{\omega}{\pi}, \quad \mathfrak{c}_{2} \gets \dfrac{\cot{\omega}}{\pi - 2\omega}$
\State $\mu_{n} \gets -\sqrt{\dfrac{\log\left(n\mathfrak{c}_{1}\right)}{\mathfrak{c}_{2}}}, \quad \sigma_{n}^{2} \gets \dfrac{-\log\log 2}{2\mathfrak{c}_{2}\left(\log\left(n\mathfrak{c}_{1}\right) - \log\log 2\right)}$
\If{$n\mathfrak{c}_{1} \leq 1$}  \Comment Check sufficient condition for the interval $\left(\infty, \mu_{n}\right]$
    \State \Return $\mathtt{False}$
\ElsIf{\eqref{eq:q-function-stoch-dom} fails over $\left[\mu_{n}, 0\right]$} \Comment Check sufficient condition for the interval $\left[\mu_{n}, 0\right]$
    \State \Return $\mathtt{False}$
\ElsIf{\eqref{eq:right-interval-stoch-dom} fails} \Comment Check sufficient condition for the interval $\left[0, \infty\right)$
    \State \Return $\mathtt{False}$
\Else
    \State \Return $\mathtt{True}$
\EndIf
\EndFunction
\end{algorithmic}
\end{algorithm}

\begin{algorithm}[!htb]
\caption{Implementation of lower bounds in Theorem \ref{thm:constructed-lower-bound} and \eqref{eq:success-prob-copula-lb-optimised}}
\label{alg:lb}
\begin{algorithmic}[1]
\Function{LowerBound}{$n$, $\alpha$, $\rho$, $\omega$} \Comment{Lower bound in \eqref{eq:success-prob-copula-lb}}
\If{\Call{\textsc{NumericalCert}}{$n$, $\omega$}}
\State \Return Right-hand side of \eqref{eq:success-prob-copula-lb}
\Else
\State \Return $0$
\EndIf
\EndFunction
\Function{OptimisedLowerBound}{$n$, $\alpha$, $\rho$} \Comment{Optimised lower bound in \eqref{eq:success-prob-copula-lb-optimised}}
\State $\epsilon \gets$ machine epsilon
\State \Return $\max_{\omega \in \left[\epsilon, \pi/2 - \epsilon\right]}$\Call{\textsc{LowerBound}}{$n$, $\alpha$, $\rho$, $\omega$}
\EndFunction
\end{algorithmic}
\end{algorithm}

Algorithm \ref{alg:optimimised_n_high_prob} implements an optimised $n^{*}$ to address Problem \ref{prob:inversion}\ref{prob:inversion_n} for the Gaussian copula success probability, as described in Section \ref{sec:invert_lower_bound}. The procedure also calls the numerical certificate from Algorithm \ref{alg:test}.

\begin{algorithm}[!htb]
\caption{Implementation of optimised $n^{*}$ for solving Problem \ref{prob:inversion}\ref{prob:inversion_n} with Gaussian copula}
\label{alg:optimimised_n_high_prob}
\begin{algorithmic}[1]
\Function{nStar}{$\alpha$, $\rho$, $\delta$, $\omega$}
\State Solve quartic \eqref{eq:quartic} for greatest real root $\mathsf{x}$ 
\State $n^{*} \gets \left\lfloor \exp\left(\mathsf{x}^{2} - \log\mathfrak{c}_{1}\right)\right\rfloor $
\If{\Call{\textsc{NumericalCert}}{$n^{*}$, $\omega$}}
\State \Return $n^{*}$
\Else
\State \Return $\infty$
\EndIf
\EndFunction
\Function{nStarOptimised}{$\alpha$, $\rho$, $\delta$}
\State $\epsilon \gets$ machine epsilon
\State \Return $\min_{\omega \in \left[\epsilon, \pi/2 - \epsilon\right]}$\Call{\textsc{nStar}}{$\alpha$, $\rho$, $\delta$, $\omega$}
\EndFunction
\end{algorithmic}
\end{algorithm}

\end{appendices}

\end{document}